\documentclass[12pt,a4wide,reqno]{amsart}
\usepackage{amsfonts, amssymb,amsmath, amsthm, amsxtra, latexsym, mathrsfs, amscd}
\usepackage{amsmath, amssymb, amsthm, times, float}
\usepackage{mathtools, amssymb, amsthm, times, float}
\usepackage{graphics}
\usepackage{MnSymbol}
\usepackage{enumerate}
\usepackage{enumitem}
\usepackage{tikz-cd}

\usepackage{amssymb,hyperref}
\usepackage{backref}
\usepackage[latin1]{inputenc} %

\newtheorem{theo+}{Theorem}[section]
\newtheorem{prop+}[theo+]{Proposition}
\newtheorem{coro+}[theo+]{Corollary}
\newtheorem{lemm+} [theo+]{Lemma}
\newtheorem{deep+}  [theo+]  {Deep Result}
\newtheorem{fact+}  [theo+]  {Fact}

\theoremstyle{definition}
\newtheorem{exam+}  [theo+]  {Example}
\newtheorem{rema+}  [theo+]  {Remark}
\newtheorem{defi+}  [theo+]  {Definition}
\newtheorem*{definition*}{Definition}
\newtheorem{xca+}[theo+]{Exercise}
\newtheorem{thmalph}{Theorem}
\newtheorem{coralph}[thmalph]{Corollary}

\newenvironment{theorem}{\begin{theo+}}{\end{theo+}}
\newenvironment{proposition}{\begin{prop+}}{\end{prop+}}
\newenvironment{corollary}{\begin{coro+}}{\end{coro+}}
\newenvironment{lemma}{\begin{lemm+}}{\end{lemm+}}

\newenvironment{remark}{\begin{rema+}}{\end{rema+}}
\newenvironment{definition}{\begin{defi+}}{\end{defi+}}

\numberwithin{equation}{section}

\usepackage{color}

\newcommand\beq{\begin{equation}\label}
\newcommand\eeq{\end{equation}}

\newcommand{\C}{\mathbb{C}}
\newcommand{\R}{\mathbb{R}}

\newcommand{\ad}{\mathrm{ad}}

\renewcommand\c[1]{{\check{#1}}}

\renewcommand\a[1]{{\acute{#1}}}

\def\draft{\centerline{(Draft {\the \day}/{\the\month} \the \year.)}}
\bigskip

\def\refn#1.#2{\expandafter\def\csname#1\endcsname{[#2]}}
\def\refnr#1.{\csname#1\endcsname}

\def\fa{\mathfrak a}

\def\fg{\mathfrak g}

\def\fk{\mathfrak k}

\def\fm{\mathfrak m}
\def\fn{\mathfrak n}
\def\fp{\mathfrak p}

\def\fsl{\mathfrak {sl}}
\def\ft{\mathfrak t}

\def\fsu{\mathfrak{su}}

\def\a{\alpha}

\def\Claminv2{|C(\Lambda)|^{-2}}

\def\de{d\varepsilon}

\def\Aa2D{A^{\a,2}(D)}
\def\bAa2D{\overline{A^{\a,2}(D)}}
\def\Ab2D{A^{\beta,2}(D)}
\def\bAb2D{\overline{A^{\beta,2}(D)}}

\def\Norm#1_#2{\Vert#1\Vert_{#2}}

\def\phipl12{\phi_{p_{l_1}, p_{l_2}}}
\def\phip01{\phi_{p_{0}, p_{0}}}

\def\a{\alpha}

\def\Claminv2{|C(\Lambda)|^{-2}}

\def\ad{\operatorname{ad}}

\def\Ind{\operatorname{Ind}}

\def\exp{\operatorname{exp}}

\def\sgn{\operatorname{sgn}}

\def\End{\operatorname{End}}

\def\de{d\varepsilon}

\def\Aa2D{A^{\a,2}(D)}
\def\bAa2D{\overline{A^{\a,2}(D)}}
\def\Ab2D{A^{\beta,2}(D)}
\def\bAb2D{\overline{A^{\beta,2}(D)}}

\def\phipl12{\phi_{p_{l_1}, p_{l_2}}}
\def\phip01{\phi_{p_{0}, p_{0}}}

\def\alg/{algebra}

\def\Alg/{Algebra}

\def\alt/{alternative} 
\def\anal/{analytic}
\def\analfunc/{\anal/\ \func/}
\def\Ans/{\it Answer. \normal}
\def\ass/{associative}
\def\nass/{non-\ass/}
\def\autom/{automorphism}
\def\homom/{homomorphism}
\def\isom/{isomorphism}
\def\bdd/{bounded}
\def\Bdd/{Bounded}
\def\bddsymdom/{bounded \sym/ \dom/}
\def\Cartdom/{Cartan \dom/}
\def\bdry/{boundary}
\def\bsd/{\bdd/ \symdom/}
\def\bv/{boundary value}
\def\cf/{{\it cf}\.}
\def\Cf/{{\it Cf}\.}
\def\charr/{character}
\def\coeff/{coefficient}
\def\comm/{commutative}
\def\cpct/{compact}
\def\compl/{complex}
\def\comp/{complex}
\def\Comp/{Complex}
\def\conf/{conformal}
\def\conj/{conjugate}
\def\conn/{connect}
\def\cont/{continuous}
\def\conv/{converge} 
\def\convc/{convergence}
\def\convt/{convergent}
\def\convx/{convex}
\def\coord/{coordinate}
\def\lcoord/{local coordinate}
\def\Corr/{Corresponding}
\def\corr/{corresponding}
\def\corrd/{correspond}
\def\cov/{covariant}
\def\decomp/{decomposition}
\def\deco/{decompose}
\def\diff/{different} 
\def\Diff/{Different} 
\def\dimn/{dimension} 
\def\distr/{distribution} 
\def\div/{diverge} 
\def\dom/{domain}
\def\eg/{\hbox{\it e.g}\.}
\def\eigenf/{eigen\-\func/}
\def\eigensp/{eigen\-space}
\def\eigenv/{eigen\-value}
\def\eq/{equation}
\def\equa/{equation}
\def\de/{\diff/ial \equa/}
\def\do/{\diff/ial operator}
\def\ode/{ordinary \de/}
\def\pde/{partial \de/}
\def\pdo/{partial \diff/ial operator}
\def\psdo/{pseudo \diff/ial operator}
\def\fin/{finite}
\def\Ex/{\it Example.\ \normal}
\def\Exnr#1/{\it Example #1.\ \normal}
\def\foll/{follow}
\def\follg/{following}
\def\Follg/{Following}
\def\func/{function}
\def\Func/{Function}
\def\Fonc/{Fonc\-tion}
\def\fonc/{fonc\-tion}
\def\Funk/{Funk\-tion}
\def\funk/{Funk\-tion}
\def\gen/{general}
\def\har/{harmonic}
\def\Hint/{\it Hint. \normal}
\def\hist/{historic}
\def\histcl/{historical}
\def\hol/{holo\-morphic}
\def\homog/{ho\-mo\-ge\-ne\-ous}
\def\hyp/{hyper\-bolic}
\def\hyperg/{hyper\-geometric}
\def\ie/{\hbox{\it i.e.}}
\def\iff/{if and only if}
\def\ineq/{inequality}
\def\infra/{{\it inf\-ra}}
\def\ultra/{{\it ult\-ra}}
\def\Inpart/{In particular}
\def\inpart/{in particular}
\def\instof/{instead of}
\def\interps/{interpolation space}
\def\interp/{interpolation}
\def\Interp/{Interpolation}
\def\interpr/{Interpretation}
\def\Intr/{Introduction}
\def\intv/{interval}
\def\inv/{invariant}
\def\invc/{invariance}
\def\Iowords/{In other words}
\def\iowords/{in other words}
\def\ipr/{inner product}
\def\irred/{irreducible}
\def\lb/{line bundle}
\def\lin/{linear}
\def\lhs/{left hand side}
\def\rhs/{right hand side}
\def\loc/{local}
\def\math/{mathematic}
\def\mathcn/{\math/ian}
\def\manif/{manifold}
\def\meas/{measure}
\def\measl/{measurable}
\def\mero/{mero\-morphic}
\def\mon/{monomial}
\def\monog/{monogenic}
\def\mult/{multiple}
\def\multy/{multiply}
\def\multn/{multiplication}
\def\nas/{necessary and sufficient}
\def\nbd/{neighborhood}
\def\neg/{negative}
\def\nondeg/{nondegenerate}
\def\Oohand/{On the other hand}
\def\oohand/{on the other hand}
\def\Oonhand/{On the one hand}
\def\oonhand/{on the one hand}
\def\oper/{operator}
\def\orth/{ortho\-gonal}
\def\orthon/{ortho\-normal}
\def\otoh/{on the other hand}
\def\quat/{quaternion}
\def\pp/{\hbox{a. e.}}
\def\psh/{plurisubharmonic}
\def\pol/{polynomial}
\def\pot/{potential}
\def\pos/{positive}
\def\princ/{principle}
\def\prob/{probability}
\def\proj/{projective}
\def\projn/{projection}
\def\Proof/{\it Proof:\normal}
\def\Rem/{\it Remark\normal}
\def\Remnr#1/{\it Remark\ \normal #1. }
\def\rep/{representation}
\def\reps/{representations}
\def\meta/{metaplectic representation}
\def\repr/{reproducing}
\def\reprker/{reproducing kernel}
\def\resp/{respective} 
\def\resply/{respectively}
\def\restr/{restriction}
\def\sa/{self-adjoint}
\def\st/{such that}

\def\sol/{solution}
\def\ru/{space}
\def\sph/{spherical}
\def\ssp/{sub\ru/}
\def\sym/{symmetric}
\def\Sym/{Symmetric}
\def\symb/{symbol}
\def\symbc/{symbolic}
\def\symdom/{\sym/ domain}
\def\symp/{symplectic}
\def\Theor#1/{\fet Theorem #1.\ \normal}
\def\Lem#1/{\fet Lemma #1.\ \normal}
\def\Lemma/{\fet Lemma.\ \normal}
\def\topl/{topology}
\def\topll/{topological}
\def\transf/{transform}
\def\transl/{translation}
\def\transfn/{transformation}
\def\transv/{transvectant}
\def\trig/{trigonometric}
\def\tril/{trilinear}
\def\trilf/{trilinear form}
\def\uhp/{upper halfplane}
\def\uhs/{upper halfspace}
\def\vb/{vector bundle}
\def\vf/{vector field}
\def\vsp/{vector space}
\def\wrt/{with respect to}
\def\Wlog/{Without loss of generality}
\def\a{\alpha}

\def\Ab/{Abel}
\def\Ban/{Banach}
\def\Bansp/{\Ban/ space}
\def\Belt/{Bel\-tra\-mi}
\def\Berg/{Berg\-man}
\def\Bern/{Ber\-nou\-lli}
\def\Berz/{Berezin}
\def\Bess/{Bessel}
\def\Cart/{Car\-tan}
\def\Cay/{Cay\-ley}
\def\CG/{Clebsch-Gordan}
\def\Cl/{Clifford}
\def\CR/{Cauchy-Rie\-mann}
\def\Dir/{Dirichlet}
\def\Eucl/{Euclide}
\def\Eucln/{Euclidean}
\def\F/{Fourier}
\def\Hank/{Hankel}
\def\Hankf/{\Hank/ form}
\def\Herm/{Hermite}
\def\Hilb/{Hilbert}
\def\Hilbs/{Hilbert space}
\def\Hilbsp/{Hilbert space}
\def\HS/{Hilbert-Schmidt}
\def\Lag/{La\-grange}
\def\Lap/{La\-place}
\def\LapBelt/{\Lap/-\Belt/}
\def\Leb/{Lebesgue}
\def\Marc/{Mar\-cin\-kie\-wicz}
\def\Moeb/{Moebius}
\def\Moebt/{Moebius transformation}
\def\Moebtransfn/{Moebius transformation}
\def\Pla/{Plan\-che\-rel}
\def\Poin/{Poin\-car\'e}
\def\Riem/{Rie\-mann}
\def\Riemn/{\Riem/ian}
\def\psRiemn/{pseudo-\Riem/ian}
\def\Riems/{Rie\-mann surface}
\def\Schroe/{Schr\"odinger}
\def\Weier/{Weier\-strass}
\def\id{\text{id}}
\def\anal/{analytic}
\def\bsd/{bounded symmetric domain  }
\def\bdd/{bounded}
\def\calc/{calculation}\def\conj{conjugate}
\def\calci/{calculating}\def\eg{e.g.}
\def\conj/{conjugate}
\def\deco/{decomposition}
\def\eg/{e.g.}
\def\fct/{function}
\def\gp/{group}
\def\hw/{highest weight}
\def\hwv/{highest weight vector}
\def\hwvs/{highest weight vectors}
\def\lw/{lowest weight}
\def\lwv/{lowest weight vector}
\def\lwvs/{lowest weight vectors}
\def\hds/{holomorphic discrete series}
\def\iff/{if and only if}
\def\inv/{invariant}
\def\irrde/{irreducible decomposition}
\def\meas/{measure}
\def\transf/{transform}
\def\rep/{representation}
\def\resp/{respectively}
\def\inters/{intertwines}
\def\interg/{intertwining}
\def\meta/{metaplectic representation}
\def\qu/{quaternion}
\def\rep/{representation}
\def\symdom/{ symmetric domain}
\def\st/{such that}
\def\shd/{subhead}
\def\transf/{transform}
\def\wrt/{with respect to}

\def\Norm#1#2#3{\Vert#1\Vert^{#3}_{{#2}}}

\DeclareMathOperator{\SL}{SL}
\let\sl\relax
\DeclareMathOperator{\sl}{\mathfrak{sl}}

\DeclareMathOperator{\SU}{SU}
\DeclareMathOperator{\su}{\mathfrak{su}}

\newcommand{\fraka}{\mathfrak{a}}

\newcommand{\frakg}{\mathfrak{g}}
\newcommand{\frakh}{\mathfrak{h}}

\newcommand{\frakk}{\mathfrak{k}}
\newcommand{\frakl}{\mathfrak{l}}

\newcommand{\frakp}{\mathfrak{p}}

\newcommand{\frakt}{\mathfrak{t}}

\newcommand{\CC}{\mathbb{C}}

\newcommand{\RR}{\mathbb{R}}
\newcommand{\ZZ}{\mathbb{Z}}

\newcommand{\calH}{\mathcal{H}}

\newcommand{\0}{\textbf{0}}

\DeclareMathOperator{\RC}{RC}

\renewcommand{\min}{{\textup{min}}}

\begin{document}
	
\title[Degenerate principal series of $G_{2(2)}$]{On the degenerate principal series of $G_{2(2)}$ induced from a Heisenberg parabolic subgroup}

\begin{abstract}
We study degenerate principal series representations of the split real group $G_{2(2)}$ induced from a character of a maximal parabolic subgroup whose unipotent radical is a Heisenberg group. Using the Lie algebra action on the space of $K$-finite vectors, we find the points of reducibility and the complementary series. The minimal representation and a limit of discrete series are identified as kernel of the corresponding Knapp--Stein intertwining operator. Moreover, we show that some quaternionic discrete series representations occur as the subrepresentation on which the family of intertwining operators vanishes of order two.
\end{abstract}



\author{Jan Frahm}
\address{Department of Mathematics, Aarhus University, Ny Munkegade 118, 8000 Aarhus C, Denmark}
\email{frahm@math.au.dk}

\author{Robin van Haastrecht}
\author{Clemens Weiske}
\author{Genkai Zhang}
\address{Mathematical Sciences, Chalmers University of Technology and Mathematical Sciences, G\"oteborg University, SE-412 96 G\"oteborg, Sweden}
\email{robinva@chalmers.se}
\email{clemens.weiske@gmail.com}
\email{genkai@chalmers.se}

\thanks{The first named author was supported by a research grant from
  the Villum Foundation (Grant No. 00025373) and a research grant from the Aarhus University Research Foundation (grant no. AUFF-E-2022-9-34). The third named author
  was supported by a research grant from the Knut and Alice Wallenberg
  foundation (KAW 2020.0275). The fourth named author was supported
  partially by the Swedish Research Council (VR, Grants 2018-03402,  2022-02861).}

\maketitle

\section*{Introduction}

Heisenberg parabolic subgroups $P$ in semisimple Lie groups $G$ and the corresponding homogeneous spaces $G/P$ are of interest both in contact geometry~\cite{HM98,Mok08}, analysis~\cite{Kab12a} and representation theory~\cite{BKZ08,Fra22,Wei03,Zha22,Zha23}. One of their appearances in representation theory is through parabolically induced representations. The family of representations of $G$ induced from characters of $P$ is often referred to as a \emph{degenerate principal series}. This paper is concerned with the detailed study of degenerate principal series representations for the  real Lie group $G_{2(2)}$, the split real form of the  complex group of type $G_2$.

In a series of papers~\cite{Zha22, Zha23}, the last named author studied degenerate principal series representations induced from Heisenberg parabolic subgroups for Hermitian, quaternionic and split exceptional Lie groups of type $F_4$, $E_6$, $E_7$ and $E_8$ by computing explicitly the action of the Lie algebra $\mathfrak g$ on the underlying $(\mathfrak g_{\mathbb C}, K)$-module, where $K$ is a maximal compact subgroup of $G$. The methods used in these works rely on a transitivity property of the Weyl group of $K$ on representations in certain tensor product decompositions. Unfortunately, this property does not hold for $G=G_{2(2)}$.

In this paper, we treat the missing case of $G_{2(2)}$ by computing explicitly the relevant tensor product decompositions using Rankin--Cohen operators. This allows us to answer the following questions about the corresponding degenerate principal series representations of $G_{2(2)}$:
\begin{itemize}
    \item For which parameters are the representations reducible?
    \item For which parameters are the representations unitarizable?
    \item Which representations occur as subrepresentations/quotients at points of reducibility?
    \item How do the Knapp--Stein intertwining operators act on each $K$-isotypic component?
\end{itemize}
In particular, we find the corresponding complementary series, the minimal representation of $G_{2(2)}$, a limit of discrete series and some quaternionic discrete series representations as subrepresentations. We expect these results to have applications in the study of exceptional theta correspondences, where $G_{2(2)}$ occurs as one member of a dual pair in an exceptional group of type $F_4$, $E_6$, $E_7$ or $E_8$.

We now describe our results in more detail.

\subsection*{Statement of the results}

The split simple real Lie group $G=G_{2(2)}$ has a unique (up to conjugation) parabolic subgroup $P$ whose unipotent radical $N$ is a $5$-dimensional Heisenberg group. Let $P=MAN$ be a Langlands decomposition, then $M$ has two connected components and we denote by $\sgn:M\to M/M_0\simeq\{\pm1\}$ the non-trivial character of $M$. The characters $\chi_s$ of $A$ are parameterized by $s\in\CC$, and we form the degenerate principal series representations
$$ \pi_{\varepsilon,s} = \Ind_P^G(\sgn^\varepsilon\otimes\chi_s\otimes1) \qquad (\varepsilon\in\ZZ/2\ZZ=\{0, 1\}, s\in\CC). $$
The parameterization is chosen such that $\pi_{\varepsilon,s}$ is unitary for $s\in\frac{3}{2}+i\RR$ and that $\pi_{0,s}$ contains the trivial representation as a subrepresentation resp. quotient for $s=0$ resp. $s=3$.

\begin{thmalph}[see Propositions~\ref{prop:irredtriv} and \ref{prop:sgnirred}]\label{thmintroA}
    The representation $\pi_{\varepsilon,s}$ is reducible if and only if
    \begin{itemize}
        \item $\varepsilon=0$ and $s\in\ZZ\cup(2+\frac{2}{3}\ZZ_{\geq0})\cup(1-\frac{2}{3}\ZZ_{\geq0})$, or
        \item $\varepsilon=1$ and $s\in\ZZ\cup(\frac{7}{3}+\frac{2}{3}\ZZ_{\geq0})\cup(\frac{2}{3}-\frac{2}{3}\ZZ_{\geq0})$.
    \end{itemize}
\end{thmalph}

By the general theory of Knapp and Stein, there exists a meromorphic family $A(\varepsilon,s)$ of intertwining operators $\pi_{\varepsilon,s}\to\pi_{\varepsilon,3-s}$. Together with the irreducibility of $\pi_{\varepsilon,s}$ for $s\in(1,2)$ this immediately yields the complementary series:

\begin{coralph}[see Corollary~\ref{cor:ComplementarySeries}]\label{corintroB}
    For $\varepsilon\in\ZZ/2\ZZ$ and $s\in\RR$ the representation $\pi_{\varepsilon,s}$ is irreducible and unitarizable if and only if $s\in(1,2)$.
\end{coralph}

By embedding $\pi_{\varepsilon,s}$ into a principal series representation induced from a minimal parabolic subgroup of $G$, these complementary series representations can be identified with irreducible subrepresentations/quotients of the principal series at the edges of its complementary series as obtained by Vogan~\cite{Vog94}.
See Section
\ref{coml-series-sec} below.

It was observed by Kable~\cite[Theorem 4.8]{Kab12} that the minimal representation of $G$ is a unitarizable subrepresentation of $\pi_{1,s}$ for $s=\frac{2}{3}$, and it is given by the joint kernel of a system of conformally invariant differential operators. This representation is also referred to as \emph{ladder representation} since its $K$-types are contained in a single line. Kable also showed in \cite[Theorem 4.10]{Kab12} that the kernel of another system of conformally invariant differential operators in $\pi_{0,s}$ for $s=\frac{1}{3}$ is a non-unitary subrepresentation $\pi_{\mathrm{small}}$ whose $K$-types are contained in two parallel lines, which is why we call it \emph{double ladder representation}. We show that both subrepresentations are equal to the kernel of the corresponding intertwining operator $A(\varepsilon,s)$, when suitably normalized.

\begin{thmalph}[see Propositions~\ref{prop:ladderker} and \ref{prop:doubleladderker}]\label{thmintroC}
    \begin{itemize}
        \item The ladder representation $\pi_{\mathrm{min}}$ of $G$ is the kernel of the intertwining operator $A(1,s)$ at $s=\frac{2}{3}$.
        \item The double ladder representation $\pi_{\mathrm{small}}$ of $G$ is the kernel of the intertwining operator $A(0,s)$ at $s=\frac{1}{3}$.
    \end{itemize}
\end{thmalph}

Finally, we identify some quaternionic discrete series representations of $G$ as subrepresentations of $\pi_{\varepsilon,s}$. The fact that these are subrepresentations was observed earlier by Yoshinaga~\cite{Yoshi98}. We relate them to the kernel of the intertwining operator $A(\varepsilon,s)$.

Recall from \cite{GW96} a family of irreducible unitary representations $\pi_k^{\textup{GW}}$, $k\geq2$. For $k\geq5$ they belong to the discrete series and for $k=4$ the representation is a limit of discrete series.

\begin{thmalph}[see Theorems~\ref{thm:LDS} and \ref{thm:QDS}]\label{thmintroD}
    \begin{itemize}
        \item The limit of discrete series $\pi_4^{\textup{GW}}$ is isomorphic to the kernel of the intertwining operator $A(\varepsilon,s)$ for $s=2$, $\varepsilon\equiv1\pmod2$.
        \item The quaternionic discrete series representation $\pi^{\textup{GW}}_k$ ($k\geq6$ even) is isomorphic to the subrepresentation of $\pi_{\varepsilon,s}$ for $s=\frac{k}{2}$, $\varepsilon\equiv\frac{k}{2}+1\pmod2$, consisting of those vectors on which the intertwining operator $A(\varepsilon,s')$ vanishes of order two at $s'=s$.
    \end{itemize}
\end{thmalph}

\subsection*{Method of proof}

The restriction of the degenerate principal series representations $\pi_{\varepsilon,s}$ to $K$ is not multiplicity-free. This makes it difficult to study the structure of the underlying $(\mathfrak{g}_\CC,K)$-module, because the transition between two $K$-types by the Lie algebra action is expressed in terms of a linear map between the multiplicity spaces. Building on the work of Kable~\cite{Kab12}, we determine these maps explicitly (see Proposition~\ref{Mmatrix}) and identify a convenient basis of every multiplicity space that is adapted to the transition between $K$-types by the Lie algebra action (see Section~\ref{sec:MultSpaceBasis}). It turns out that the action of the Knapp--Stein intertwining operators on each $K$-isotypic component is upper triangular with respect to this basis, and we are able to determine the corresponding eigenvalues (see Corollary~\ref{cor:ClosedFormulasIntertwinerMultOneKtypes} and Theorem~\ref{thm:ClosedFormulasIntertwinerEigenvalues}). This allows us to prove Theorem~\ref{thmintroA} and deduce Corollary~\ref{corintroB}. The proof of Theorems~\ref{thmintroC} and \ref{thmintroD} also uses the explicit form of the eigenvalues together with a detailed analysis of the transition between certain $K$-types.
As far as we know
our  method
of constructing
 bases where 
 the Knapp-Stein
 operator acts as upper triangular matrices has
 not appeared before, and it may be useful
 for studying other principal
 series.

\subsection*{Other related works}

This work builds on Kable's paper~\cite{Kab12} where an abstract formula for the Lie algebra action was obtained and where the first vector in our basis of the multiplicity spaces appears. We also view our paper as an attempt to treat the last missing case in the algebraic study of degenerate principal series begun by the last named author~\cite{Zha22,Zha23}. More analytic aspects of the same class of degenerate principal series can be found in \cite{FWZ22}, where the Knapp--Stein operators were studied using the Heisenberg group Fourier transform, and in \cite{Fra22}, where an $L^2$-model for the the minimal representation is constructed also using the Heisenberg group Fourier transform.

\section{Preliminaries}

We introduce some notation for the group $G=G_{2(2)}$, its Lie algebra $\mathfrak{g}=\mathfrak{g}_{2(2)}$, the Heisenberg parabolic subgroup and the corresponding degenerate principal series representations.

\subsection{The Lie algebra $\fg=\fg_{2(2)}$ }
\label{LieAlgPrel}

Let $\fg=\fg_{2(2)}$ denote the split real form of the complex simple Lie algebra of type $G_2$. We introduce some notation, following the conventions in \cite{Kab12}.

Let $\fg_2$ be the complexification of $\fg_{2(2)}$, then we know that with respect to a Cartan subalgebra $\frakh \subseteq \fg_2$ the Lie algebra $\fg_2$ has two simple roots $\alpha$ and $\beta$, with $\alpha$ short and $\beta$ long. We denote the Killing form by $\langle -,-\rangle$ and normalize it such that $\langle \alpha,\alpha \rangle = 2$ and $\langle \beta, \beta \rangle = 6$. Then $\langle \alpha, \beta \rangle = -3$. The highest root is $3 \alpha + 2 \beta$.

We choose a Cartan decomposition $\fg_{2(2)} = \frakk \oplus \frakp$, and we let $\frakh \subseteq \fg_2$ be a Cartan subalgebra such that $\fraka_{\frakp} = \frakh \cap \fg_{2(2)} \subseteq \frakp$ is
a real Cartan subalgebra. Let $\Delta^+$ be a set of positive roots, and for $\nu \in \Delta$ we let $X_{\nu}, H_{\nu}, X_{-\nu}$ be a corresponding $\fsl_2(\RR)$-triple, where $\theta(X_{\mu}) = -X_{-\mu}$ and $H_\nu$ is the co-root of $\nu$. We let $Z_{\mu} = X_{\mu} - X_{-\mu}\in\mathfrak k
$ and $W_{\mu} = X_{\mu} + X_{-\mu}
\in\mathfrak p
$, then the set $\{ Z_{\mu} \ | \ \mu \in \Delta^+ \}$ forms a basis for $\frakk$ and the set $\{ W_{\mu} \ | \ \mu \in \Delta^+ \} \cup \{H_{\alpha}, H_{\beta} \}$ forms a basis of $\frakp$.

We identify the subalgebra $\frakk$ with $\su(2) \oplus \su(2)$ in the following way. We take the triple $$ \begin{pmatrix} i & 0 \\ 0 & -i
\end{pmatrix}, \quad \begin{pmatrix} 0 & 1 \\ -1 & 0
\end{pmatrix}, \quad \begin{pmatrix} 0 & i \\ i & 0
\end{pmatrix} $$ 
as a  "standard"
$\su(2)$-triple. We obtain two $\su(2)$-triples $U_1,U_2,U_3$ and $V_1,V_2,V_3$ in $\frakk$,
matching the 
commutation relations
of the standard 
$\su(2)$-triple above,
where
\begin{equation*}
 U_1 = \frac{1}{2} (3Z_{3\alpha+2\beta} + Z_{\alpha}), \quad U_2 = \frac{1}{2} (Z_{\alpha+\beta} - 3 Z_{3 \alpha + \beta}), \quad U_3 = - \frac{1}{2} (3 Z_{\beta} + Z_{2 \alpha + \beta})
\end{equation*}
and
\begin{equation*}
 V_1 = \frac{1}{2}(Z_{3\alpha + 2 \beta} - Z_{\alpha}), \quad V_2 = \frac{1}{2} (Z_{\alpha + \beta} + Z_{3 \alpha + \beta}), \quad V_3 = \frac{1}{2} (Z_{\beta} - Z_{2 \alpha + \beta}),
\end{equation*}
as in \cite[pp. 100--101]{Kab12}. This specifies an explicit isomorphism $
\frakk=\su(2)_s\oplus
\su(2)_l \cong\su(2)\oplus\su(2)
$, the complexification
of $\su(2)_s$
having the short roots
and  $\su(2)_l$
having the
long roots.
We fix this
realization and
let $K=\SU(2) \times \SU(2) / \{ \pm (I,I) \}$ acting
on $\mathfrak p$ as real (irreducible) linear transformation. We further fix the Cartan subalgebra $\frakt=\RR U_1\oplus\RR V_1$. With respect to this, the complexification of $\frakp$ is isomorphic to the representation $\CC^4 \otimes \CC^2$ of $K$, with highest weight vector $\frac{1}{2}(H_{3 \alpha + 2 \beta} + i W_{3 \alpha + 2 \beta})$.

The group $M_{\frakp} = Z_{\fraka_{\mathfrak{p}}}(K)$, the centralizer of $\fraka_{\mathfrak{p}}$ in $K$, is generated by the elements $m_{\mu} = \exp(\pi Z_{\mu})$ for $\mu \in \Delta^+$ (see \cite[Corollary 2.14]{Vog94}). Under the above identification $K\cong(\SU(2)\times\SU(2))/\{\pm(I,I)\}$ we find that
\begin{align*}
    m_{\alpha} = m_{3 \alpha + 2 \beta} &= \left[ \begin{pmatrix} i & 0 \\ 0 & -i \end{pmatrix}, \begin{pmatrix} i & 0 \\ 0 & -i \end{pmatrix}\right],\\
    m_{\alpha+\beta} = m_{3 \alpha + \beta} &= \left[ \begin{pmatrix} 0 & -1 \\ 1 & 0 \end{pmatrix}, \begin{pmatrix} 0 & 1 \\ -1 & 0 \end{pmatrix}\right],\\
    m_{2 \alpha + \beta} = m_{\beta} &= \left[ \begin{pmatrix} 0 & i \\ i & 0 \end{pmatrix}, \begin{pmatrix} 0 & -i \\ -i & 0 \end{pmatrix}\right],
\end{align*} 
where $[(g,h)]$ denotes the coset of $(g,h)\in\SU(2)\times\SU(2)$ in $K$.

\subsection{The Heisenberg parabolic subalgebra}

The element $E = H_{3 \alpha + 2 \beta}\in\mathfrak{a}_{\mathfrak p}$ has the property that $\ad(E)$ has eigenvalues $0,\pm 1, \pm 2$ on $\fg$. The direct sum of eigenspaces with non-negative eigenvalues is a parabolic subalgebra $\fm+\fa+\fn$ of $\fg$, where $\fa=\mathbb R E$, $\fm$ its orthogonal complement within the $0$-eigenspace of $\ad(E)$, and $\fn=\fn_1+\fn_2$ the sum of the $1$- and $2$-eigenspaces. Note that $\fn_1$ is $4$-dimensional and $\fn_2$ is $1$-dimensional, so $\fn$ is a Heisenberg algebra. Consequently, the half sum of positive roots with respect to $\fa$ is given by $$\rho_{\fg}=3 E^\ast =3,
$$
where $E^\ast\in
\mathfrak a^\ast\simeq \mathbb R$ is  dual to $E$. Moreover, $\fm\simeq\sl(2,\RR)$ acts on $\fn_1$ by its $4$-dimensional irreducible real representation and on $\fn_2$ trivially.

\subsection{Groups and induced representations}
Let $G=G_{2(2)}$ be the adjoint group of $\frakg_{2(2)}$. Its maximal compact subgroup is $K$ above. 
We write $MAN$ for the parabolic subgroup of $G$ with Lie algebra $\fm+\fa+\fn$. Clearly, $N=\exp(\fn)$ is a Heisenberg group and $A=\exp(\fa)$ is one-dimensional. Moreover, $M$ has two connected components, and we denote by $\sgn$ the non-trivial character of $M$.

For $\varepsilon\in\ZZ/2\ZZ, s\in \mathbb C$ and $\nu=\nu(s)=2sE^*\in(\fa_\CC)^*$ we consider the degenerate principal series representation
$$ \pi_{\varepsilon,s} = \Ind_{MAN}^G(\sgn^\varepsilon\otimes\exp(\nu)\otimes 1) $$
realized on the space
\begin{equation}
    \label{def-ind} 
I(\varepsilon,s) = 
\{f:G\to\CC:
\text{$f$ measurable}, f|_K\in L^2(K),f(gman)=\sgn(m)^\varepsilon a^{-\nu}f(g)\}. 
\end{equation}
The representations are unitary for $\nu\in\rho_{\fg}+i\fa^*$ which corresponds to $s\in\frac{3}{2}+i\RR$. We note that the parameter $s$ is the same as in \cite{Kab12}.

\subsection{$K$-types}

To determine the $K$-types of $I(\varepsilon,s)$, we first find the $K$-types of
$$ \pi_s = \Ind_{M_0AN}^G(1\otimes \exp(\nu) \otimes1) \qquad \mbox{on} \qquad I(s) = I(0,s) \oplus I(1,s) $$
and then determine which of them belong to $I(0,s)$ or $I(1,s)$. By restricting functions from $G$ to $K$, we find that, as $K$-representation, $I(s)|_K\simeq L^2(K/L_0)$, where $L_0$ is the identity component of $L=M\cap K$. We note that the Lie algebra of $L_0$ is given by $\frakl_0 = \R (U_1 - 3 V_1)$. The decomposition of $L^2(K/L_0)$ into irreducible representations of $K$ is given by the Peter--Weyl Theorem in terms of the matrix coefficients
$$ \psi_\sigma(\xi_1,\xi_2):K\to\CC, \quad \psi_\sigma(\xi_1,\xi_2)(k)=\langle \sigma(k)\xi_2,\xi_1\rangle_\sigma, $$
where $\xi_1,\xi_2\in E_\sigma$ are vectors in a unitary representation $(\sigma,E_\sigma)$ of $K$ equipped with a
fixed inner product $\langle-,-\rangle_\sigma$. More precisely, the map
$$ \bigoplus_{\sigma\in\widehat{K}}\overline{E_\sigma}\otimes  E_\sigma^{L_0} \to L^2(K/L_0), \quad \overline{\xi_1}\otimes\xi_2 \mapsto\psi_\sigma(\xi_1,\xi_2) $$
is a $K$-equivariant isomorphism onto the dense subspace of $K$-finite vectors in $L^2(K/L_0)$. Here $\overline{E_\sigma}$ denotes the complex conjugate vector spaces and $E_\sigma^{L_0}$ stands
for the subspace of $L_0$-fixed vectors.

Let $\Gamma_{m}$ denote the irreducible representation of $\SU(2)$ of dimension $m+1$. The irreducible representations of $K=(\SU(2) \times \SU(2) )/\{\pm(I,I)\}$ are of the form $\sigma(n,m)=\Gamma_{n}\boxtimes\Gamma_{m}$ with $n\equiv m\pmod2$.
Let $\{\xi_a^n\}_{-n\leq a\leq n,a\equiv n\pmod2}\subseteq\Gamma_n$ denote the basis of $\Gamma_n$ as in Appendix~\ref{appendix} (see also \cite[p.~101]{Kab12}). For $a\equiv m\equiv n\pmod2$ we let
$$ \zeta_a^{n,m} = \xi_{3a}^n\boxtimes\xi_a^m \in \Gamma_n\boxtimes\Gamma_m. $$
Then by \cite[Lemma 4.5]{Kab12} the set
$$ \{\zeta_a^{n,m}:3|a|\leq n,|a|\leq m,a\equiv n\equiv m\pmod2\} $$
is a basis of $(\Gamma_n\boxtimes\Gamma_m)^{L_0}$. 
Denote, for $n,m \in \ZZ_{\geq 0}$ with $n\equiv m\pmod2$,
\begin{equation}
    \label{a-n-m}
    a(n,m):=\dim \left(
(\Gamma_n\boxtimes \Gamma_m)^{L_0}
\right)-1.
\end{equation}
Explicitly we have
$$ a(n,m)=
\begin{cases}
	\min([\frac{n}{3}]-1,m) & \text{if $n\equiv1\pmod3$,} \\
	\min([\frac{n}{3}],m) & \text{else.}
\end{cases}$$
Then  $a(n,m)+1$ is
the multiplicity of the $K$-type $\Gamma_n\boxtimes \Gamma_m$ in $L^2(K/L_0)$.
This yields the $K$-type decomposition of $I(s)=I(0,s)\oplus I(1,s)$.

To deduce from this the $K$-type decompositions of $I(0,s)$ and $I(1,s)$, note that there exists an element $w \in L$ such that $L = L_0 \cup w L_0$. Explicitly in $K =(\SU(2) \times \SU(2) )/\{\pm(I,I)\}$ we can choose
$$ w = \left[ \begin{pmatrix} 0 & 1 \\ -1 & 0
\end{pmatrix}, \begin{pmatrix} 0 & -1 \\ 1 & 0
\end{pmatrix} \right].$$
(Note that $w$ agrees with $\kappa_{2\alpha+\beta}$ in \cite[Lemma 2.3]{Kab12} up to conjugation.) Then $I(\varepsilon,s)$ consists of those $f\in I(s)$ such that $f(gw)=(-1)^\varepsilon f(g)$. In terms of the matrix coefficients $\psi_\sigma(\xi_1,\xi_2)$ with $\xi_2\in E_\sigma^{L_0}$, this translates to $\psi_\sigma(\xi_1,\xi_2)\in I(\varepsilon,s)$ if and only if $\sigma(w)\xi_2=(-1)^\varepsilon\xi_2$. Write $E_\sigma^{L,\varepsilon}$ for the subspace of $E_\sigma^{L_0}$ that has this property. Using the formulas in Appendix~\ref{appendix}, it is easy to see that $w$ acts on the basis vectors $\zeta_a^{n,m}$ of $(\Gamma_n\boxtimes\Gamma_m)^{L_0}$ as
$$ w \cdot \zeta_a^{n,m} = (-1)^{\frac{n-m}{2}} \zeta_{-a}^{n,m} = (-1)^{\frac{m+n}{2} + a} \zeta_{-a}^{n,m}, $$
where we have used that $a\equiv m\pmod2$.
Hence $(\Gamma_n\boxtimes\Gamma_m)^{L,\varepsilon}$ has as a basis the vectors 
\begin{equation}\label{eq:MultiplicitySpaceWeightBasis}
\zeta_a^{n,m} + (-1)^{\frac{m+n}{2} + a + \varepsilon} \zeta_{-a}^{n,m}
\qquad (0 \leq a \leq a(n,m), a \equiv a(n,m) \pmod 2),
\end{equation}
where as a convention the zero vector is omitted.

\begin{remark}
    We note that the $K$-type decomposition is written in the same way as in \cite[Section 3]{Zha23}. The group $L_1$ in \cite{Zha23} is the maximal torus with Lie algebra $\ft$ and the parameter $p$ in \cite{Zha23} is called $a$ in this paper (following the notation of \cite{Kab12}).
\end{remark}

\section{The action of $\mathfrak{p}_{\mathbb{C}}$}

In the previous section we have described the action of $K$ on the degenerate principal series representations $I(\varepsilon,s)$. In view of the Cartan decomposition $\fg=\fk\oplus\fp$, it suffices to understand the action of $\fp$ in order to answer questions about reducibility and unitarity. In this section we therefore study the action of $\fp$, or rather its complexification $\fp_\CC$, on the $K$-finite vectors of $I(s)=I(0,s)\oplus I(1,s)$. Later, by distinguishing between $K$-finite vectors in $I(0,s)$ and $I(1,s)$, we obtain results about the two degenerate principal series representations.

\subsection{General formulas for the action of $\mathfrak{p}_\CC$ on $\Gamma_n \boxtimes \Gamma_m$}

Recall from the previous section that the $K$-finite vectors in $I(s)$ are spanned by the matrix coefficients $\psi_\sigma(\xi_1,\xi_2)$, where $\xi_1\in E_\sigma=\Gamma_n\boxtimes\Gamma_m$ with $n\equiv m\pmod2$ and $\xi_2\in(\Gamma_n\boxtimes\Gamma_m)^{L_0}$ which is spanned by the vectors $\zeta_a^{n,m}$, $3|a|\leq n$, $|a|\leq m$, $a\equiv n\equiv m\pmod2$. In \cite[page 112]{Kab12}, Kable introduced a map
$$ R_s(\sigma): 
\Gamma_n \boxtimes \Gamma_m \rightarrow \mathfrak{p}_\CC \otimes (\Gamma_n \boxtimes \Gamma_m),\,
\sigma=\sigma(n, m),
$$
such that for $Y \in \mathfrak{p}_\CC$, $\xi_1 \in \Gamma_n \boxtimes \Gamma_m$ and $\xi_2 \in (\Gamma_n \boxtimes \Gamma_m)^{L_0}$:
\begin{equation}\label{eq:PactionViaRmap}
    d\pi_s(Y) \psi_{\xi_1, \xi_2} = - \psi_{\overline{Y} \otimes \xi_1,R_s(\sigma) \xi_2}.
\end{equation}
(Compared to \cite{Kab12}, we use the complex conjugate matrix coefficients, so that \eqref{eq:PactionViaRmap} becomes complex linear in $\xi_2$ and the right hand side depends holomorphically on $s$.) Moreover, in \cite[Lemma 4.7]{Kab12} he gave an explicit formula for $R_s(\sigma)$ on the vectors $\zeta_a^{n,m}$:
$$ R_s(\sigma(n,m)) = R_s^+(n,m)+R_s^-(n,m), $$
where
\begin{align*}
	R_s^+(n,m)\zeta_a^{n,m}:={}&(a+s)(\xi_3^3\boxtimes\xi_1^1 )\otimes\zeta_a^{n,m}
	\\&+\frac{n-3a}{2}(\xi_1^3\boxtimes\xi_1^1)\otimes (\xi_{3a+2}^n\boxtimes\xi_a^m)
	\\&+\frac{m-a}{2}(\xi_3^3\boxtimes \xi_{-1}^1)\otimes(\xi_{3a}^n\boxtimes\xi_{a+2}^m),
\end{align*}
and
\begin{align*}
	R_s^-(n,m)\zeta_a^{n,m}:={}&(a-s)(\xi_{-3}^{3}\boxtimes\xi_{-1}^1) \otimes\zeta_a^{n,m}
	\\&-\frac{n+3a}{2}(\xi_{-1}^3\boxtimes\xi_{-1}^1)\otimes (\xi_{3a-2}^n\boxtimes\xi_a^m)
	\\&-\frac{m+a}{2}(\xi_{-3}^3\boxtimes \xi_{1}^1)\otimes(\xi_{3a}^n\boxtimes\xi_{a-2}^m).
\end{align*}

Since this action involves tensor products of representations of $\SU(2)$, we first make their decomposition into irreducibles explicit in terms of the Rankin--Cohen brackets. Abstractly, we know that
$$ \Gamma_n\otimes\Gamma_m \simeq \bigoplus_{k=0}^{\min(n,m)}\Gamma_{n+m-2k}. $$
An explicit isomorphism is given in terms of the Rankin--Cohen brackets $\RC_k=\RC_{n, m; k}:\Gamma_n\otimes \Gamma_m \to \Gamma_{n+m-2k}$, which we normalize as in Appendix~\ref{appendix}.

\begin{proposition}
\label{Mmatrix}
	We have
    $$(\RC_{l}\boxtimes \RC_{l'})R_s^{\pm}(n,m)\zeta^{n,m}_{a}=M^{\pm}_{l,l'}\zeta^{n+3-2l,m+1-2l'}_{a\pm 1}, \qquad (0\le l\le 3, 0\le l'\le 1), $$
	where the matrices $M^\pm=(M^\pm_{l,l'})_{0\leq l\leq 3,0\leq l'\leq1}$ 
    are given by
	$$M^+:=\left(
	\begin{array}{cc}
		\frac{1}{2} (2s -2 a+m+n) & \frac{(a-m) (2s - 2 a - m-2 + n)}{4 m} \\
		\frac{(3 a-n) (6s - 6a -6 + 3 m + n)}{12 n} & \frac{(a-m) (3 a-n) (6 s - 6 a - 12 +n-3 m)}{24 m n} \\
		\frac{(3 a-n) (n - 3a - 2) (6 a- 6 s - 3 m+n+8)}{24 (n-1) n} & \frac{(a-m) (n - 3 a) (n - 3 a-2) (6s - 6 a - 14 - 3 m - n)}{48 m (n-1) n} \\
		\frac{(n - 3a) (n - 3 a - 2) (n - 3 a - 4) (2 a - 2s -m+n+2)}{16 (n-2) (n-1) n} & \frac{(a-m) (3 a-n) (n - 3 a - 2) (n - 3 a - 4) (2s - 2 a - 4 - m - n)}{32 m (n-2) (n-1) n} \\
	\end{array}
	\right)
    $$
    and
    $$M^-:=	\left( \begin{array}{cc}
		- \frac{1}{2} (2s + 2 a+m+n) & \frac{(-a-m) (2s + 2 a - m-2 + n)}{4 m} \\
		\frac{(- 3 a-n) (6s + 6a -6 + 3 m + n)}{12 n} & \frac{(a+m) (-3 a-n) (6 s + 6 a - 12 +n-3 m)}{24 m n} \\
		\frac{(3 a+n) (n + 3a - 2) (-6 a- 6 s - 3 m+n+8)}{24 (n-1) n} & \frac{(-a-m) (n + 3 a) (n + 3 a-2) (6s + 6 a - 14 - 3 m - n)}{48 m (n-1) n} \\
		\frac{(n + 3a) (n + 3 a - 2) (n + 3 a - 4) (- 2 a - 2s -m+n+2)}{16 (n-2) (n-1) n} & \frac{(a+m) (- 3 a-n) (n + 3 a - 2) (n + 3 a - 4) (2s + 2 a - 4 - m - n)}{32 m (n-2) (n-1) n} \\
	\end{array}
	\right).$$
\end{proposition}
\begin{proof}
	Apply the explicit formulas for the Rankin--Cohen brackets from Appendix~\ref{appendix}.
\end{proof}

\subsection{A convenient basis}\label{sec:MultSpaceBasis}

Recall $a(n, m)$
defined in (\ref{a-n-m}).
The vector $\zeta^{n,m}_{a(n,m)}$ is the highest $L_0$-invariant weight vector in $\Gamma_n \boxtimes \Gamma_m$.
Let 
\begin{equation}
\label{r(n,m)}
r(n,m):=\frac{n+m}{2}-2a(n,m).    
\end{equation}
For $s\in\CC$, $a \equiv a(n,m) \pmod 2$ and $k\in\ZZ_{\geq0}$ we define coefficients
\begin{multline*}
    \quad \quad 
    d_{n,m}(s,a,k):=(-1)^{\frac{a(n,m)-a}{2}-k}\binom{a(n,m)-2k}{\frac{a(n,m)-a}{2}-k}\\
    \times\frac{\Gamma(\frac{1+r(n,m)+s}{2}+2k)\Gamma(\frac{1+r(n,m)+s}{2}+a(n,m))}{\Gamma(\frac{1+r(n,m)+s}{2}+k+\frac{a(n,m)-a}{2})\Gamma(\frac{1+r(n,m)+s}{2}+k+\frac{a(n,m)+a}{2})},
\end{multline*}
where we use the notation
$$ {x\choose y}=\frac{\Gamma(x+1)}{\Gamma(y+1)\Gamma(x-y+1)} \qquad (x,y\in\CC). $$
These coefficients are used to define two families of $(\frakk \cap \frakl)$-invariant vectors $v_{n,m}$ and $v'_{n,m}$ for $k\in \ZZ_{\geq 0}$ with $2k\leq a(n,m)$ by
\begin{align*}
    v_{n,m}(s,k) &:= \sum_{\substack{|a| \leq a(n,m)\\a \equiv a(n,m) \pmod 2}} d_{n,m}(s,a,k)\zeta_{a}^{n,m},\\
    v_{n,m}'(s,k) &:= \sum_{\substack{|a| \leq a(n,m)\\a \equiv a(n,m) \pmod 2}} \frac{a}{a(n,m)-2k}d_{n,m}\left(s+1,a,k\right)\zeta_{a}^{n,m}.
\end{align*}
Note that for $0\leq 2k\leq a(n,m)$ and $a\equiv a(n,m)\pmod2$ we have
$$ d_{n,m}(s,a,k)=0 \qquad \mbox{if }|a|>a(n,m)-2k, $$
so that the summation is over $2k - a(n,m)\leq a\leq a(n,m) - 2k$. In this case, we can write

\begin{equation*}
\begin{split}
&{\quad}
\frac{\Gamma(\frac{1+r(n,m)+s}{2}+2k)\Gamma(\frac{1+r(n,m)+s}{2}+a(n,m))}{\Gamma(\frac{1+r(n,m)+s}{2}+k+\frac{a(n,m)-a}{2})\Gamma(\frac{1+r(n,m)+s}{2}+k+\frac{a(n,m)+a}{2})}
\\
 &=\frac{(\frac{1+r(n,m)+s}{2}+k+\frac{a(n,m)+|a|}{2})_{\frac{a(n,m)-|a|}{2}-k}}{(\frac{1+r(n,m)+s}{2}+2k)_{\frac{a(n,m)-|a|}{2}-k}},
\end{split}
\end{equation*}
so $d_{n,m}(s,a,k)$ has simple poles at $s=-1-r(n,m)-4k,-1-r(n,m)-4k-2,\ldots,1-r(n,m)-2k-a(n,m)+|a|$, which are all negative integers. It follows that the $v_{n,m}(s,k)$ are holomorphic in $s \in \CC\setminus(\ZZ_{\leq-1} \cap (\frac{m+n+2}{2} + 2 \ZZ))$ and the $v_{n,m}'(s,k)$ are holomorphic in $s\in \CC\setminus(\ZZ_{\leq-1} \cap (\frac{m+n}{2} + 2 \ZZ))$.

\begin{lemma}
\label{basemultspace}
    Let $n,m\in\ZZ_{\geq0}$ with $n\equiv m\pmod2$ and put $\varepsilon=\frac{m+n}{2}+2\ZZ\in\ZZ/2\ZZ$. Then for $s \in \CC\setminus(\ZZ_{\leq-1} \cap (\frac{m+n+2}{2} + 2 \ZZ))$ the set
    $$ \{v_{n,m}(s,k):0\leq k\leq[\tfrac{a(n,m)}{2}]\} $$
    is a basis of $(\Gamma_n \boxtimes \Gamma_m)^{L,\varepsilon}$, and for $s\in \CC\setminus(\ZZ_{\leq-1} \cap (\frac{m+n}{2} + 2 \ZZ))$ the set
    $$ \{v'_{n,m}(s,k'):0\leq k'\leq[\tfrac{a(n,m)-1}{2}]\} $$
    is a basis of $(\Gamma_n \boxtimes \Gamma_m)^{L, \varepsilon+1}$. In particular, we obtain bases of the multiplicity spaces of all $K$-types in $I(0,s)$ whenever $s \in \CC\setminus(\ZZ_{\leq-1} \cap 2\ZZ+1)$, and similar for $I(1,s)$ whenever $s\in\CC\setminus(\ZZ_{\leq-1}\cap2\ZZ)$.
\end{lemma}

\begin{proof}
    By \eqref{eq:MultiplicitySpaceWeightBasis} we know that
    $$ \dim\left((\Gamma_n\boxtimes\Gamma_m)^{L,\delta}\right) = \begin{cases}1+[\frac{a(n,m)}{2}]&\mbox{if }\delta\equiv\frac{m+n}{2}\pmod2,\\1+[\frac{a(n,m)-1}{2}]&\mbox{if }\delta\equiv\frac{m+n+2}{2}\pmod2.\end{cases} $$
    Since $d_{n,m}(s,a,k) = (-1)^{a} d_{n,m}(s,-a,k)$ it follows that $v_{n,m}(s,k)\in(\Gamma_n \boxtimes \Gamma_m)^{L,\varepsilon}$ and $v_{n,m}'(s,k')\in(\Gamma_n \boxtimes \Gamma_m)^{L,\varepsilon+1}$, where we let $\varepsilon\equiv\frac{m+n}{2}\pmod2$. Therefore, we only need to check linear independence of the vectors $\{v_{n,m}(s,k)\}$ and $\{v_{n,m}'(s,k')\}$.
	Clearly we have
    $$ d_{n,m}(s,a,k)=\frac{a}{a(n,m)-2k}d_{n,m}(s+ 1,a,k)=0 \qquad \mbox{whenever }|a|>a(n,m)-2k, $$
    while
    $$ d_{n,m}(s,a(n,m)-2k,k)=\frac{a(n,m)-2k}{a(n,m)-2k}d_{n,m}(s+1,a(n,m)-2k,k)=1 $$
    and
    $$ d_{n,m}(s,2k-a(n,m),k) = (-1)^{m} = -\frac{2k-a(n,m)}{a(n,m)-2k}d_{n,m}(s+1,2k-a(n,m),k). $$
    This implies
    \begin{align*}
        v_{n,m}(s,k) &= \zeta_{a(n,m)-2k}^{n,m}+(-1)^m\zeta_{2k-a(n,m)}^{n,m}+\sum_{|a|<a(n,m)-2k}(\cdots)\zeta_a^{n,m},\\
        v_{n,m}'(s,k) &= \zeta_{a(n,m)-2k}^{n,m}+(-1)^{m+1}\zeta_{2k-a(n,m)}^{n,m}+\sum_{|a|<a(n,m)-2k}(\cdots)\zeta_a^{n,m},
    \end{align*}
    so the claim follows.
\end{proof}

\begin{prop+}\label{prop:v_action} For all $(n,m)$ we have:
	\begin{align}
	    (\RC_0\otimes \RC_0)R_s(\sigma(n,m))v_{n,m}(s,k) &= (2k+r(n,m)+a(n,m)+s)v_{n+3,m+1}(s,k),\label{prop:v_action1}\\
	   (\RC_0\otimes \RC_0)R_s(\sigma(n,m))v'_{n,m}(s,k) &= (2k+r(n,m)+a(n,m)+s)v'_{n+3,m+1}(s,k).\label{prop:v_action2}
    \end{align}
	For all $(n,m)$ with $a(n+1,m+1)=a(n,m)+1$ we have:
	\begin{multline}
		(\RC_1\otimes \RC_0)R_s(\sigma(n,m))v_{n,m}(s,k)=\\-\frac{(6k-3a(n,m)+n)(3s+r(n,m)-a(n,m)+m+6k-3)}{6n}v'_{n+1,m+1}(s,k),\label{prop:v_action3}
	\end{multline}
	and
	\begin{multline}
		(\RC_1\otimes \RC_0)R_s(\sigma(n,m))v'_{n,m}(s,k)=\\-\frac{(6k-3a(n,m)+n)(3s+r(n,m)-a(n,m)+m+6k-3)}{6n}v_{n+1,m+1}(s,k)\\
		-\frac{2(6k-3 a(n,m)+n+3)(2k+r(n,m)+a(n,m)+s)(2k+r(n,m)+a(n,m)+s+1)}{3n (4k+r(n,m)+s) (4 k+r(n,m)+s+2)} \\\times  (3s+r(n,m)-a(n,m)+m+6k)
		v_{n+1,m+1}(s,k+1).\label{prop:v_action4}
	\end{multline}
For all $(n,m)$ with $a(n+3,m-1)=a(n,m)+1$ we have:
\begin{multline}
	(\RC_0\otimes \RC_1)R_s(\sigma(n,m))v_{n,m}(s,k)=\\ \frac{(2k+m-a(n,m))(2s+n-m-2a(n,m)+4k-2)}{4m}v'_{n+3,m-1}(s,k),\label{prop:v_action5}
\end{multline}
and
	\begin{multline}
	(\RC_0\otimes \RC_1)R_s(\sigma(n,m))v'_{n,m}(s,k)=\\ \frac{(2k+m-a(n,m))(2s+n-m-2a(n,m)+4k-2)}{4m}v_{n+3,m-1}(s,k)\\
	-\frac{(2k+m-a(n,m)+1)(2s+n-m-2a(n,m)+4k)(2s+n+m-2a(n,m)+4k)}{m(2s+n+m-4a(n,m)+8k)(2s+n+m-4a(n,m)+8k+4)} \\\times  (2s+n+m-2a(n,m)+4k+2)
	v_{n+3,m-1}(s,k+1).\label{prop:v_action6}
\end{multline}
\end{prop+}
\begin{proof}
	First note that $a(n+3,m+1)=a(n,m)+1$ for all $n,m \in \ZZ_{\geq 0}$.
	We start with the first statement. We have
	\begin{multline*}
		(\RC_0\otimes \RC_0)R_s(\sigma(n,m))v_{n,m}(s,k)= \\ \sum_{\substack{|a| \leq a(n,m)\\a \equiv a(n,m) \pmod 2}}\Big(d_{n,m}(s,a,k)(\RC_0\otimes \RC_0)R^+_s(n,m)\zeta_{a}^{n,m}\\
        + d_{n,m}(s,a,k)(\RC_0\otimes \RC_0)R^-_s(n,m)\zeta_{a}^{n,m}\Big)
	\end{multline*}
which is by Proposition~\ref{Mmatrix} equal to
\begin{multline*}
	\sum_{\substack{|a| \leq a(n,m)\\a \equiv a(n,m) \pmod 2}} \bigg( \left(s+\frac{n+m}{2}-a \right) d_{n,m}(s,a,k)\zeta_{a + 1}^{n+3,m+1}\\ + \left(-s-\frac{n+m}{2}-a \right)d_{n,m}(s,a,k)\zeta_{a - 1}^{n+3,m+1}\bigg).
\end{multline*}
Changing the summation from $a$ to $a\pm1$ and using that $a(n+3,m+1)=a(n,m)+1$, we find
\begin{multline*}
		\sum_{\substack{|a| \leq a(n+3,m+1)\\a \equiv a(n+3,m+1) \pmod 2}} \bigg( \left(s+\frac{n+m}{2}-(a-1) \right) d_{n,m}(s,a-1,k)\\ + \left(-s-\frac{n+m}{2}-(a+1) \right)d_{n,m}(s,a+1,k)\bigg)\zeta_{a}^{n,m}.
\end{multline*}
An easy calculation using the definition of $d_{n,m}(s,a,k)$ shows that
\begin{multline*}
	\left(s+\frac{n+m}{2}-(a-1) \right) d_{n,m}(s,a-1,k) + \left(-s-\frac{n+m}{2}-(a+1) \right)d_{n,m}(s,a+1,k)\\=(2k+r(n,m)+a(n,m)+s)d_{n+3,m+1}(s,a,k).
\end{multline*}
The other equations follow in the same way by applying the explicit formulas of Proposition \ref{Mmatrix} for $R_s(\sigma(n,m))\zeta_a^{n,m}$.
\end{proof}

For $n\equiv m\pmod2$ and $s\in\CC$ we consider in $(\Gamma_n\boxtimes\Gamma_m)^{L_0}$ 
the ordered basis
\begin{equation}\label{eq:basis_ordering}
	\mathcal{B}_s^{n,m}:=(v_{n,m}(s,0),v'_{n,m}(s,0),v_{n,m}(s,1),v'_{n,m}(s,1), \dots).
\end{equation}
In the following section we use this basis to determine the transition matrices of the Knapp--Stein intertwining operators between $K$-isotypic components with respect to this basis.

\begin{remark}
In \cite[page 130]{Kab12}, Kable defines for every $m\in\ZZ_{\geq0}$ and $r\in\{0,1,2\}$ and all $\lambda\in\CC$ a vector $u_m(r,\lambda)\in(\Gamma_{3m+2r}\boxtimes\Gamma_m)^{L_0}$ by
$$ u_m(r,\lambda) = \sum_{\substack{|a| \leq m\\a \equiv m \pmod 2}}(-1)^{\frac{m-a}{2}}\binom{m}{\frac{m-a}{2}}\frac{\Gamma(\lambda)\Gamma(\lambda+m)}{\Gamma(\lambda+\frac{m-a}{2})\Gamma(\lambda+\frac{m+a}{2})}\zeta_{a}^{3m+2r,m}. $$
Since $a(3m+2r,m)=m$, this coincides with our first basis vector $v_{3m+2r,m}(s,0)$ for $\lambda=\frac{1+r+s}{2}$. In fact, Kable's formula was our inspiration to find a basis for the whole multiplicity space $(\Gamma_{3m+2r}\boxtimes\Gamma_m)^{L_0}$ that behaves well with respect to the action of the Knapp--Stein intertwining operators. The action of the Knapp--Stein operators on this basis is studied in the next section.
\end{remark}

\section{Knapp--Stein intertwining operators}

We study the action of the Knapp--Stein intertwining operator on each $K$-isotypic component in the degenerate principal series.

\subsection{The intertwining property}

Since the parabolic subgroup $MAN$ and its opposite are conjugate, there exists a meromorphic family of intertwining operators
$$ A(\varepsilon,s):I(\varepsilon,s)\to I(\varepsilon,3-s), $$
intertwining the representations $\pi_{\varepsilon,s}$ and $\pi_{\varepsilon,3-s}$. This family of operators is unique up to normalization by a meromorphic function of $s$. To ease notation, we first consider intertwining operators
$$ A(s):I(s)\to I(3-s). $$
Since $I(s)=I(0,s)\oplus I(1,s)$, these are linear combinations of $A(0,s)$ and $A(1,s)$.

In this section, we determine the action of $A(s)$ on the $K$-isotypic components. Since $A(s)$ is in particular $K$-intertwining, it leaves each $K$-isotypic component $\overline{E_\sigma}\otimes E_\sigma^{L_0}$ invariant. By Schur's Lemma, the restriction of $A(s)$ to $\overline{E_\sigma}\otimes E_\sigma^{L_0}$ is given by $\id_\sigma\otimes A_\sigma(s)$ for some $A_\sigma(s)\in\End(E_\sigma^{L_0})$. For $\sigma=\sigma(n,m)$ we write $A_{n,m}(s)$ for the transformation matrix of the linear map $A_{\sigma(n,m)}(s)$ with respect to the bases $\mathcal{B}_s^{n,m}$ and $\mathcal{B}_{3-s}^{n,m}$.

To determine $A_{n,m}(s)$, we use the intertwining property for the derived representation $d\pi_s$ of $\fg_\CC$:
$$
A(s)\circ d\pi_s(X) = d\pi_{3-s}(X)\circ A(s), \quad X\in \fg_{\mathbb C}.
$$
By \eqref{eq:PactionViaRmap}, this implies 
\begin{equation}\label{eq:IntertwiningForLinearMaps}
    R_{3-s}(\sigma(n,m))\circ A_{\sigma(n,m)}(s) = \sum_{n',m'}A_{\sigma(n',m')}(s)\circ R_s(\sigma(n,m)),
\end{equation}
where the sum rums over all $n',m'$ such that $\Gamma_{n'}\boxtimes\Gamma_{m'}$ occurs in the tensor product $\fp_\CC\otimes(\Gamma_n\boxtimes\Gamma_m)$. Write $T_{n,m}^{n',m'}(s)$ for the transformation matrix of the linear map
$$ (\RC_l\otimes\RC_{l'})\circ R_s(\sigma(n,m)):(\Gamma_n\boxtimes \Gamma_m)^{L_0} \to (\Gamma_{n'}\boxtimes \Gamma_{m'})^{L_0} $$
with respect to the bases $\mathcal{B}_s^{n,m}$ and $\mathcal{B}_s^{n',m'}$, where $n'=n+3-2l$ and $m'=m+1-2l'$. Then \eqref{eq:IntertwiningForLinearMaps} can be written as
\begin{equation}\label{eq:IntertwiningForMatrices}
    T_{n,m}^{n',m'}(3-s)\cdot A_{n,m}(s) = A_{n',m'}(s)\cdot T_{n,m}^{n',m'}(s), \qquad (n'\in\{n\pm1,n\pm3\}, m'\in\{m\pm1\}).
\end{equation}
Note that $T_{n,m}^{n',m'}(s)$ is essentially determined in Proposition~\ref{Mmatrix}. Some matrix entries are even determined explicitly in Proposition~\ref{prop:v_action}.

\subsection{Action on multiplicity-one $K$-types}

We first focus on $K$-types with multiplicity-one. These are the $K$-types $\sigma(n,m)$ with $m=0$ or $n\in\{0,2,4\}$, and in this case the basis $\mathcal{B}_s^{n,m}$ consists of the single vector $v_{n,m}(s,0)=\zeta_0^{n,m}$. Hence, $A_{n,m}(s)$ is a scalar. 
To obtain recurrence relations for these scalars, we first study the Lie algebra action between multiplicity-one $K$-types.

\begin{lemma}\label{LieAlg-one-mult-K-types}
For $r\in\ZZ_{\geq0}$ and $s\in\CC$:
\begin{align*}
    T_{2r+3,1}^{2r+4,0}({s})T_{2r,0}^{2r+3,1}(s) &= -\frac{2(r+3)(r+s)(r+3s-3)}{3(2r+3)},\\
    T_{3,2r+1}^{2,2r+2}({s})T_{0,2r}^{3,2r+1}(s) &= -\frac{2}{3}(r+s)(3r+3s-1),\\
    T_{5,2r+1}^{4,2r+2}({s})T_{2,2r}^{5,2r+1}(s) &= -\frac{2}{5}(r+s+1)(3r+3s-2),\\
    T_{3,2r+1}^{4,2r}({s})T_{0,2r}^{3,2r+1}(s) &= \frac{2(r+1)(r+s)(r-s+1)}{(2r+1)}.
\end{align*} 
\end{lemma}

\begin{proof}
    We show how to prove the first identity, the others follow similarly. Note that by analyticity it suffices to show the identities for generic $s$ where the vectors $v_{n,m}(s,k)$ are defined. By Proposition~\ref{prop:v_action}
    $$ (\RC_0\otimes\RC_0)\circ R_s(\sigma(2r,0))v_{2r,0}(s,0) = (r+s)v_{2r+3,1}(s,0). $$
    Since $v_{2r+3,1}(s,0)=\zeta^{2r+3,1}_1-\zeta^{2r+3,1}_{-1}$, we use Proposition~\ref{Mmatrix} to compute
    \begin{align*}
        (\RC_1\otimes\RC_1)\circ R_s(\sigma(2r+3,1))\zeta^{2r+3,1}_1 &= -\frac{(r+3)(r+3s-3)}{3(2r+3)}\zeta_0^{2r+4,0},\\
        (\RC_1\otimes\RC_1)\circ R_s(\sigma(2r+3,1))\zeta^{2r+3,1}_{-1} &= \frac{(r+3)(r+3s-3)}{3(2r+3)}\zeta_0^{2r+4,0}.
    \end{align*}
    Since $v_{2r+4,0}(s,0)=\zeta_0^{2r+4,0}$, this shows the first formula.
\end{proof}

\begin{lemma}
\label{KS-one-mult-K-types}
For $r\in\ZZ_{\geq0}$ and $s\in\CC$, $\tilde{s}=3-s$, the following identities hold:
\begin{align}
    (r+s)(r+3s-3) A_{2r+4,0}(s) &= (r+\tilde{s})(r+3\tilde{s}-3) A_{2r,0}(s),\label{eq:KSrecurrence1}\\
    (r+s)(3r+3s-1) A_{2,2r+2}(s) &= (r+\tilde{s})(3r+3\tilde{s}-1) A_{0,2r}(s),\label{eq:KSrecurrence2}\\
    (r+s+1)(3r+3s-2) A_{4,2r+2}(s) &= (r+\tilde{s}+1)(3r+3\tilde{s}-2) A_{2,2r}(s),\label{eq:KSrecurrence3}\\
    (r+s)(r-s+1) A_{4,2r}(s) &= (r+\tilde{s})(r-\tilde{s}+1) A_{0,2r}(s).\label{eq:KSrecurrence4}
\end{align}
Moreover, \eqref{eq:KSrecurrence2}--\eqref{eq:KSrecurrence4} imply
\begin{align}
    &(r+s)^2(3r+3s-1)(3r+3s+1)A_{0,2r+4} = (r+\tilde{s})^2(3r+3\tilde{s}-1)(3r+3\tilde{s}+1)A_{0,2r},\label{eq:KSrecurrence5}\\
    &(r+s-1)(r+s+1)(3r+3s-2)(3r+3s+2)A_{2,2r+4}\notag\\
    &\hspace{5cm}= (r+\tilde{s}-1)(r+\tilde{s}+1)(3r+3\tilde{s}-2)(3r+3\tilde{s}+2)A_{2,2r},\label{eq:KSrecurrence6}\\
    &(r+s-2)(r+s+2)(3r+3s-1)(3r+3s+1)A_{4,2r+4}\notag\\
    &\hspace{5cm}= (r+\tilde{s}-2)(r+\tilde{s}+2)(3r+3\tilde{s}-1)(3r+3\tilde{s}+1)A_{4,2r}.\label{eq:KSrecurrence7}
\end{align}
\end{lemma}

\begin{proof}
Identities \eqref{eq:KSrecurrence1}--\eqref{eq:KSrecurrence4} are easily verified using the intertwining identities
\begin{align*}
    A_{2r+4,0}(s)T_{2r+3,1}^{2r+4,0}({s})T_{2r,0}^{2r+3,1}(s) &= T_{2r+3,1}^{2r+4,0}(\tilde{s})T_{2r,0}^{2r+3,1}(\tilde{s})A_{2r,0}(s),\\
    A_{2,2r+2}(s)T_{3,2r+1}^{2,2r+2}({s})T_{0,2r}^{3,2r+1}(s) &= T_{3,2r+1}^{2,2r+2}(\tilde{s})T_{0,2r}^{3,2r+1}(\tilde{s})A_{0,2r}(s),\\
    A_{4,2r+2}(s)T_{5,2r+1}^{4,2r+2}({s})T_{2,2r}^{5,2r+1}(s) &= T_{5,2r+1}^{4,2r+2}(\tilde{s})T_{2,2r}^{5,2r+1}(\tilde{s})A_{2,2r}(s),\\
    A_{4,2r}(s)T_{3,2r+1}^{4,2r}({s})T_{0,2r}^{3,2r+1}(s) &= T_{3,2r+1}^{4,2r}(\tilde{s})T_{0,2r}^{3,2r+1}(\tilde{s})A_{0,2r}(s),
\end{align*}
which follow from \eqref{eq:IntertwiningForMatrices}, as well as Lemma~\ref{LieAlg-one-mult-K-types}. Combining \eqref{eq:KSrecurrence2}--\eqref{eq:KSrecurrence4} shows the remaining identities.
\end{proof}

Solving these recurrence relations yields the following formulas, which should be viewed as meromorphic identities in $s\in\CC$:

\begin{corollary}\label{cor:ClosedFormulasIntertwinerMultOneKtypes}
For $r\in\ZZ_{\geq0}$ and $s\in\CC$, $\tilde{s}=3-s$:
\begin{align*}
    A_{4r,0} &= \frac{(\frac{\tilde{s}}{2})_r(\frac{3\tilde{s}-3}{2})_r}{(\frac{s}{2})_r(\frac{3s-3}{2})_r} A_{0,0}, & A_{4r+2,0}  &= \frac{(\frac{\tilde{s}+1}{2})_r(\frac{3\tilde{s}-2}{2})_r}{(\frac{s+1}{2})_r(\frac{3s-2}{2})_r} A_{2,0},\\
    A_{0,4r} &= \frac{(\frac{\tilde{s}}{2})_r^2(\frac{3\tilde{s}-1}{6})_{r}(\frac{3\tilde{s}+1}{6})_{r}}{(\frac{s}{2})_r^2(\frac{3s-1}{6})_{r}(\frac{3s+1}{6})_{r}} A_{0,0}, &  A_{0,4r+2} &= \frac{(\frac{\tilde{s}+1}{2})_r(\frac{\tilde{s}-1}{2})_{r+1}(\frac{3\tilde{s}+2}{6})_{r}(\frac{3\tilde{s}-2}{6})_{r+1}}{   (\frac{s+1}{2})_r(\frac{s-1}{2})_{r+1}(\frac{3s+2}{6})_{r}(\frac{3s-2}{6})_{r+1}  } A_{2,0},\\
	A_{2,4r+2} &= \frac{(\frac{\tilde{s}}{2})_{r+1}(\frac{\tilde{s}}{2})_r(\frac{3\tilde{s}+1}{6})_r(\frac{3\tilde{s}-1}{6})_{r+1}}{(\frac{s}{2})_{r+1}(\frac{s}{2})_r(\frac{3s+1}{6})_r(\frac{3s-1}{6})_{r+1}} A_{0,0}, & A_{2,4r} &= \frac{(\frac{\tilde{s}+1}{2})_{r}(\frac{\tilde{s}-1}{2})_r(\frac{3\tilde{s}+2}{6})_r(\frac{3\tilde{s}-2}{6})_{r}}{(\frac{s+1}{2})_{r}(\frac{s-1}{2})_r(\frac{3s+2}{6})_r(\frac{3s-2}{6})_{r}} A_{2,0},\\
    A_{4,4r} &= \frac{(\frac{\tilde{s}}{2})_{r-1}(\frac{\tilde{s}}{2})_{r+1}(\frac{3\tilde{s}-1}{6})_r(\frac{3\tilde{s}+1}{6})_r}{(\frac{s}{2})_{r-1}(\frac{s}{2})_{r+1}(\frac{3s-1}{6})_r(\frac{3s+1}{6})_r}A_{0,0}, & A_{4,4r+2} &= \frac{(\frac{\tilde{s}+1}{2})_{r+1}(\frac{\tilde{s}-1}{2})_{r}(\frac{3\tilde{s}+2}{6})_{r}(\frac{3\tilde{s}-2}{6})_{r+1}}{(\frac{s+1}{2})_{r+1}(\frac{s-1}{2})_{r}(\frac{3s+2}{6})_{r}(\frac{3s-2}{6})_{r+1}} A_{2,0}.
\end{align*}
\end{corollary}

\begin{proof}
    The first two identities follow directly from \eqref{eq:KSrecurrence1}. The other three follow from \eqref{eq:KSrecurrence5}--\eqref{eq:KSrecurrence7} together with the following formulas which also follow from the identities in Lemma~\ref{KS-one-mult-K-types}:
    \begin{align*}
        (s-1)(3s-2)A_{0,2}(s) &= (\tilde{s}-1)(3\tilde{s}-2)A_{2,0}(s),\\
        s(3s-1)A_{2,2}(s) &= \tilde{s}(3\tilde{s}-1)A_{0,0}(s),\\
        (s+1)(3s-2)A_{4,2}(s) &= (\tilde{s}+1)(3\tilde{s}-2)A_{2,0}(s).\qedhere
    \end{align*}
\end{proof}

\begin{rema+}
    \begin{enumerate}
        \item The multiplicity-one $K$-types $\Gamma_0\boxtimes\Gamma_0$ and $\Gamma_2\boxtimes\Gamma_0$ belong to the two different degenerate principal series $I(0,s)$ and $I(1,s)$. Therefore, generically, an intertwiner $A(s):I(s)\to I(\tilde{s})$ is uniquely determined by the two eigenvalues $A_{0,0}$ and $A_{2,0}$.
        \item Observe that for $s=\frac{3}2 +i\lambda\in\frac{3}{2}+i\RR$ we have $\tilde s=\frac{3}2 -i\lambda=\bar s$ and hence all coefficients in the previous corollary are of absolute value one. This is consistent with the fact that $A(s)$ is unitary in this case.
    \end{enumerate}
\end{rema+} 

\subsection{Action on higher multiplicity $K$-types}

We now use the intertwining property to express the eigenvalues of the matrix $A_{n,m}(s)$ for an arbitrary $K$-type $\sigma(n,m)$ in terms of the eigenvalues on multiplicity-one $K$-types.

\begin{prop+}
\label{evprop}
	For all $(n,m)$, $n\equiv m\pmod2$, the matrix $A_{n,m}(s)$ is upper triangular and its diagonal entries $\mu_0(n,m,s),\dots,\mu_{a(n,m)}(n,m,s)$ satisfy the following recurrence relations:
        \begin{enumerate}
        \item For all $0\leq j\leq a(n,m)$:
            \begin{equation}\label{eq:eigenvalue_recursion_1}
		          \mu_j(n+3,m+1,s) =\frac{2[\frac{j}{2}]+r(n,m)+a(n,m)+\tilde{s}}{2[\frac{j}{2}]+r(n,m)+a(n,m)+s}\mu_j(n,m,s).
	       \end{equation}
        \item If $a(n+1,m+1)=a(n,m)+1$:
	       \begin{multline}
	         \qquad\qquad\mu_{a(n,m)+1}(n+1,m+1,s)=\mu_{a(n,m)}(n,m,s)\frac{3m+n-6+6\tilde{s}}{3m+n-6+6s}\\
              \times\begin{cases}\displaystyle\frac{(m+n+2-2\tilde{s})(m+n-2+2\tilde{s})}{(m+n+2-2s)(m+n-2+2s)}&\mbox{if $a(n,m)$ is odd,}\\1&\mbox{if $a(n,m)$ is even.}\end{cases}\label{eq:eigenvalue_recursion_2}
           \end{multline}
        \item If $a(n+3,m-1)=a(n,m)+1$:
	       \begin{multline}
	           \qquad\qquad\mu_{a(n,m)+1}(n+3,m-1,s)=\mu_{a(n,m)}(n,m,s)\frac{n-m-2+2\tilde{s}}{n-m-2+2s}\\
               \times\begin{cases}\displaystyle\frac{(m+n+2-2\tilde{s})(m+n-2+2\tilde{s})}{(m+n+2-2s)(m+n-2+2s)}&\mbox{if $a(n,m)$ is odd,}\\1&\mbox{if $a(n,m)$ is even.}\end{cases}\label{eq:eigenvalue_recursion_3}
            \end{multline}
        \end{enumerate}
\end{prop+}
\begin{proof}
	We first prove by induction over $a(n,m)$ that $A_{n,m}(s)$ is upper triangular. First note that $a(n,m)=0$ only in the cases
    $$ (n,m)=(2r,0),(0,2r),(2,2r),(4,2r). $$
    Then $A_{n,m}(s)$ is just a scalar, so clearly upper triangular. For the induction step, consider the intertwining relation
    $$ T_{n,m}^{n+3,m+1}(\tilde{s})A_{n,m}(s) = A_{n+3,m+1}(s)T_{n,m}^{n+3,m+1}(s). $$
    Here, $T_{n,m}^{n+3,m+1}(s)$ and $T_{n,m}^{n+3,m+1}(\tilde{s})$ are matrices of size $(a(n,m)+1)\times a(n,m)$ whose last row is zero and whose upper $a(n,m)\times a(n,m)$ block is a diagonal matrix with diagonal entries given by the coefficients in \eqref{prop:v_action1} and \eqref{prop:v_action2}. Since the diagonal entries are non-zero for generic $s$, the intertwining relation implies that if $A_{n,m}(s)$ is upper triangular, then $A_{n+3,m+1}(s)$ must be upper triangular as well, at least for generic $s$. But the entries of $A_{n+3,m+1}(s)$ are meromorphic in $s$, so it must be upper triangular for all $s$. This completes the induction. Moreover, the intertwining relation together with the formulas \eqref{prop:v_action1} and \eqref{prop:v_action2} for the entries of $T_{n,m}^{n+3,m+1}(s)$ and $T_{n,m}^{n+3,m+1}(\tilde{s})$ shows (1).
    
    To prove (2) we use the intertwining relation
	$$ T_{n,m}^{n+1,m+1}(\tilde{s})A_{n,m}(s)=A_{n+1,m+1}(s)T_{n,m}^{n+1,m+1}(s) $$
    together with the formulas \eqref{prop:v_action3} and \eqref{prop:v_action4} for the last row and column of $T_{n,m}^{n+1,m+1}(s)$ and $T_{n,m}^{n+1,m+1}(\tilde{s})$. Similarly, one can show (3) by using the intertwining relation
	$$ T_{n,m}^{n+3,m-1}(\tilde{s})A_{n,m}(s)=A_{n+3,m-1}(s)T_{n,m}^{n+3,m-1}(s) $$
    as well as the formulas \eqref{prop:v_action5} and \eqref{prop:v_action6} for the last row and column of $T_{n,m}^{n+3,m-1}(s)$ and $T_{n,m}^{n+3,m-1}(\tilde{s})$.
\end{proof}
We immediately obtain the following Theorem.

\begin{theorem}\label{thm:ClosedFormulasIntertwinerEigenvalues}
	The eigenvalues $\mu_k(n,m,s)$ of $A_{n,m}(s)$ are given in terms of the eigenvalues on multiplicity-one $K$-types as follows, where $m,r\in\ZZ_{\geq0}$:
    \begin{enumerate}
        \item The eigenvalues of $A_{3m+2r,m}(s)$ are given by
        \begin{align*}
            \mu_{2j}(3m+2r,m,s) &=
            \frac{(\frac{r+3\tilde{s}+2j-3}{2})_{2j}(\frac{r+\tilde{s}}{2}+j)_j(\frac{r-\tilde{s}}{2}+j+1)_{j}(r+\tilde{s}+4j)_{m-2j}}{(\frac{r+3s+2j-3}{2})_{2j}(\frac{r+s}{2}+j)_j(\frac{r-s}{2}+j+1)_{j}(r+s+4j)_{m-2j}}A_{2(r+2j),0}(s),\\
            \mu_{2j+1}(3m+2r,m,s) &=
            \frac{(\frac{r+3\tilde{s}+2j-2}{2})_{2j+1}(\frac{r+\tilde{s}+1}{2}+j)_{j}(\frac{r-\tilde{s}+3}{2}+j)_{j}(r+\tilde{s}+4j+1)_{m-2j-1}}{(\frac{r+3s+2j-2}{2})_{2j+1}(\frac{r+s+1}{2}+j)_{j}(\frac{r-s+3}{2}+j)_{j}(r+s+4j+1)_{m-2j-1}}A_{2(r+2j+1),0}(s).
	   \end{align*}
       \item The eigenvalues of $A_{3m,m+2r}(s)$ are given by
       \begin{align*}
           \mu_{2j}(3m,m+2r,s) &= \frac{(\frac{\tilde{s}-r-2j-1}{2})_{2j}(\frac{r+\tilde{s}}{2}+j)_j(\frac{r-\tilde{s}}{2}+j+1)_j(r+\tilde{s}+4j)_{m-2j}}{(\frac{s-r-2j-1}{2})_{2j}(\frac{r+s}{2}+j)_j(\frac{r-s}{2}+j+1)_j(r+s+4j)_{m-2j}} A_{0,2(r+2j)}(s),\\
           \mu_{2j+1}(3m,m+2r,s) &= \frac{(\frac{\tilde{s}-2j-r-2}{2})_{2j+1}(\frac{r+\tilde{s}+1}{2}+j)_j(\frac{r-\tilde{s}+3}{2}+j)_j(r+\tilde{s}+4j+1)_{m-2j-1}}{(\frac{s-2j-r-2}{2})_{2j+1}(\frac{r+s+1}{2}+j)_j(\frac{r-s+3}{2}+j)_j(r+s+4j+1)_{m-2j-1}} A_{0,2(r+2j+1)}(s).
       \end{align*}
       \item The eigenvalues of $A_{3m+2,m+2r}(s)$ are given by
       \begin{align*}
           \mu_{2j}(3m+2,m+2r,s) &= \frac{(\frac{\tilde{s}-r-2j}{2})_{2j}(\frac{r+\tilde{s}+1}{2}+j)_j(\frac{r-\tilde{s}+3}{2}+j)_j(r+\tilde{s}+4j+1)_{m-2j}}{(\frac{s-r-2j}{2})_{2j}(\frac{r+s+1}{2}+j)_j(\frac{r-s+3}{2}+j)_j(r+s+4j+1)_{m-2j}}A_{2,2(r+2j)}(s),\\
           \mu_{2j+1}(3m+2,m+2r,s) &= \frac{(\frac{\tilde{s}-r-2j-1}{2})_{2j+1}(\frac{r+\tilde{s}}{2}+j+1)_j(\frac{r-\tilde{s}}{2}+j+2)_j(r+\tilde{s}+4j+2)_{m-2j-1}}{(\frac{s-r-2j-1}{2})_{2j+1}(\frac{r+s}{2}+j+1)_j(\frac{r-s}{2}+j+2)_j(r+s+4j+2)_{m-2j-1}}A_{2,2(r+2j+1)}(s).
       \end{align*}
       \item The eigenvalues of $A_{3m+4,m+2r}(s)$ are given by
       \begin{align*}
           \mu_{2j}(3m+4,m+2r,s) &= \frac{(\frac{\tilde{s}-r-2j+1}{2})_{2j}(\frac{r+\tilde{s}}{2}+j+1)_j(\frac{r-\tilde{s}}{2}+j+2)_j(r+\tilde{s}+4j+2)_{m-2j}}{(\frac{s-r-2j+1}{2})_{2j}(\frac{r+s}{2}+j+1)_j(\frac{r-s}{2}+j+2)_j(r+s+4j+2)_{m-2j}}A_{4,2(r+2j)}(s),\\
           \mu_{2j+1}(3m+4,m+2r,s) &= \frac{(\frac{\tilde{s}-r-2j}{2})_{2j+1}(\frac{r+\tilde{s}+2j+3}{2})_j(\frac{r-\tilde{s}+2j+5}{2})_j(r+\tilde{s}+4j+3)_{m-2j-1}}{(\frac{s-r-2j}{2})_{2j+1}(\frac{r+s+2j+3}{2})_j(\frac{r-s+2j+5}{2})_j(r+s+4j+3)_{m-2j-1}}A_{4,2(r+2j+1)}(s).
       \end{align*}
    \end{enumerate}
\end{theorem}

\begin{proof}
    We only show the first identity, the others follow in a similar fashion. Using \eqref{eq:eigenvalue_recursion_1} $(m-2j)$ times we find
    $$ \mu_{2j}(3m+2r,m,s) = \frac{(r+\tilde{s}+4j)_{m-2j}}{(r+s+4j)_{m-2j}}\mu_{2j}(6j+2r,2j,s). $$
    Now we cannot use \eqref{eq:eigenvalue_recursion_1} anymore, because $a(6j+2r,2j)=2j$, so $\mu_{2j}(6j+2r,2j,s)$ is the last eigenvalue of $A_{6j+2r,2j}(s)$. Instead we use the two identities in \eqref{eq:eigenvalue_recursion_2} alternately, each of them $j$ times, to find that
    $$ \mu_{2j}(6j+2r,2j,s) = \frac{(\frac{r+3\tilde{s}+2j-3}{2})_{2j}(\frac{r+\tilde{s}}{2}+j)_j(\frac{r-\tilde{s}}{2}+j+1)_{j}}{(\frac{r+3s+2j-3}{2})_{2j}(\frac{r+s}{2}+j)_j(\frac{r-s}{2}+j+1)_{j}}\mu_0(4j+2r,0,s). $$
    Since $\sigma(4j+2r,0)$ is a multiplicity-one $K$-type, we have $\mu_0(4j+2r,0,s)=A_{4j+2r,0}(s)$.
\end{proof}

\begin{remark}
    Although the numbers $\mu_j(n,m,s)$ are eigenvalues of the matrix $A_{n,m}(s)$, they might not be the eigenvalues of the Knapp--Stein operator $A(s):L^2(K/L_0)\to L^2(K/L_0)$. This is because the representation matrix $A_{n,m}(s)$ is taken with respect to two \emph{different} bases $\mathcal{B}_s^{n,m}$ and $\mathcal{B}_{\tilde{s}}^{n,m}$.
\end{remark}

\begin{remark}
The matrix $A_{n,m}(s)$ is in general not diagonal. One of the smallest examples where this happens is the case $(n,m)=(6,2)$. This can be seen by considering the transition from the $K$-type $\sigma(3,3)$ to $\sigma(6,2)$. The corresponding bases are given by
$$ \mathcal{B}_s^{3,3} = (v_{3,3}(s,0),v_{3,3}'(s,0)) \qquad \mbox{and} \qquad \mathcal{B}_s^{6,2} = (v_{6,2}(s,0),v_{6,2}'(s,0),v_{6,2}(s,1)). $$
Since $v_{3,3}(s,0)$ and $v_{3,3}'(s,0)$ belong to the two different degenerate principal series $I(1,s)$ and $I(0,s)$, we have
$$ A_{3,3}(s) = \begin{pmatrix}\mu_0(3,3,s)&0\\0&\mu_1(3,3,s)\end{pmatrix}. $$
Using \ref{prop:v_action5} and \ref{prop:v_action6} we find that
$$ T_{3,3}^{6,2}(s) = \begin{pmatrix}0&(s-2)/3\\(s-2)/3&0\\0&-2(s-1)(s+2)/(s+1)\end{pmatrix}. $$
The intertwining relation $T_{3,3}^{6,2}(\tilde{s})A_{3,3}(s)=A_{6,2}(s)T_{3,3}^{6,2}(s)$ now implies that
$A_{6,2}(s)$ has an upper triangular form
$$ A_{6,2}(s) = \begin{pmatrix}\mu_0(6,2,s)&0&\Xi\\0&\mu_1(6,2,s)&0\\0&0&\mu_2(6,2,s)\end{pmatrix}, $$
and using the identities (see Theorem~\ref{thm:ClosedFormulasIntertwinerEigenvalues})
$$ \mu_1(3,3,s) = \frac{\tilde{s}(3\tilde{s}-1)(3\tilde{s}+1)}{s(3s-1)(3s+1)} \qquad \mbox{and} \qquad \mu_0(6,2,s) = \frac{\tilde{s}(\tilde{s}+1)}{s(s+1)}, $$
a short computation shows that the constant $\Xi$ in 
$A_{6,2}(s)$ is
$$ \Xi = \frac{4\tilde{s}(2s-3)}{s(s-1)(s+2)(3s-1)(3s+1)}. $$
\end{remark}

\section{Reducibility and Unitarity}

In this section we investigate reducibility criteria for our modules and identify some of the composition factors at points of reducibility.

\subsection{Reducibility}

For $\varepsilon\in\ZZ/2\ZZ$ and $s\in\CC$, $\tilde{s}=3-s$, there is a non-degenerate invariant pairing $\langle-,-\rangle_{\varepsilon,s}:I(\varepsilon,s) \times I(\varepsilon,\tilde{s}) \rightarrow \C$ given by
\begin{equation}
\label{Gpairing}
\langle f_1, f_2 \rangle_{\varepsilon,s} = \int_{K} f_1(k) f_2(k)\,dk.
\end{equation}
Therefore, $I(\varepsilon,s)$ and $I(\varepsilon,\tilde{s})$ are contragredient to each other. We first prove an abstract irreducibility criterion.

\begin{lemma}
\label{genKtypes}
Let $V$ be a multiplicity-one $K$-type in $I(\varepsilon,s)|_K$. Then $I(\varepsilon,s)$ is irreducible if and only if $V$ generates both $I(\varepsilon,s)$ and $I(\varepsilon,\tilde{s})$.
\end{lemma}

\begin{proof}
If $I(\varepsilon,s)$ is irreducible, then obviously $V$ generates $I(\varepsilon,s)$. Now $I(\varepsilon,\tilde{s})$ is dual to $I(\varepsilon,s)$ by the above pairing \eqref{Gpairing} and is irreducible if and only if $I(\varepsilon,s)$ is, so $V$ will also generate $I(\varepsilon,\tilde{s})$.

For the other direction, assume $V$ generates both $I(\varepsilon,s)$ and $I(\varepsilon,\tilde{s})$, and let $W \subseteq I(\varepsilon,s)$ be an invariant subspace. Using the invariant pairing, we consider $W^\perp = \{ v \in I(\varepsilon,\tilde{s}) \ | \ \langle W,v \rangle_{\varepsilon,s} = 0 \}$. Since $W$ is invariant under $\pi_{\varepsilon,s}$, $W^\perp$ is invariant under $\pi_{\varepsilon,\tilde{s}}$. But $V$ has multiplicity-one, and as $W\oplus W^\perp\cong I(\varepsilon,s)|_K\cong I(\varepsilon,\tilde{s})|_K$, the $K$-type $V$ has to be entirely in $W$ or in $W^\perp$, which then has to be the whole space $I(\varepsilon,s)$ respectively $I(\varepsilon,\tilde{s})$. This implies $W = \{0\}$ or $W = I(\varepsilon,s)$, and thus $I(\varepsilon,s)$ is irreducible.
\end{proof}

Now we apply this criterion to obtain the points of reducibility for $I(\varepsilon,s)$. We treat $\varepsilon=0$ and $\varepsilon=1$ separately.

\begin{proposition}
\label{prop:irredtriv}
The representation $\pi_{0,s}$ is reducible if and only if $s$ or $\tilde{s}=3-s$ is contained in $\ZZ_{\geq2}\cup(2+\frac{2}{3}\ZZ_{\geq0})$.
\end{proposition}

\begin{proof}
We first show that if $s$ or $\tilde{s}$ is contained in $\ZZ_{\geq2}\cup(2+\frac{2}{3}\ZZ_{\geq0})$, then $\pi_{0,s}$ is reducible by showing that there is a multiplicity-one $K$-type in $I(0,s)$ on which the intertwiner $A(0,s)$ vanishes. The multiplicity-one $K$-types in $I(0,s)$ are of the form $\Gamma_n\boxtimes\Gamma_m$ with
$$ (n,m)=(0,4p),(2,4p+2),(4,4p)\mbox{ or }(4p,0), \qquad p\in\ZZ_{\geq0}. $$
It is readily checked that if $s$ or $\tilde{s}$ in $\ZZ_{\geq2}\cup(2+\frac{2}{3}\ZZ_{\geq0})$, then on the LHS or RHS
of one of the formulas \eqref{eq:KSrecurrence1}--\eqref{eq:KSrecurrence4} (with $r=2p$ in \eqref{eq:KSrecurrence1} and \eqref{eq:KSrecurrence4} and $r=2p+1$ in \eqref{eq:KSrecurrence2} and \eqref{eq:KSrecurrence3}) there is a coefficient of the $A$'s that vanishes while all coefficients on the respective
RHS or LHS are non-trivial. This shows that one of the coefficients $A_{n,m}(s)$ for $(n,m)$ a certain multiplicity-one $K$-type must vanish. But the family $A(0,s)$ of intertwining operators can be normalized so that the corresponding operator $A(0,s)$ is not identically zero. This implies that its kernel is a non-trivial proper subrepresentation and hence $\pi_{0,s}$ is reducible.

We now show that if $s$ and $\tilde{s}$ are both not contained in $\ZZ_{\geq2}\cup(2+\frac{2}{3}\ZZ_{\geq0})$, then $\pi_{0,s}$ is irreducible. By Lemma~\ref{genKtypes} it suffices to show that a multiplicity-one $K$-type generates both $I(0,s)$ and $I(0,\tilde{s}).$ We fix the generating $K$-type to be $(0,0)$ and show that it generates $I(0,s)$.

We first observe that, by the formulas of Lemma~\ref{LieAlg-one-mult-K-types}, we generate all the multiplicity-one $K$-types from $(0,0)$. Indeed, this is the case if all coefficients in Lemma~\ref{LieAlg-one-mult-K-types} are non-zero, which can easily be verified.

Next we show that all other $K$-isotypic components are generated from $(0,0)$. We prove by induction over $n$ that the isotypic component of $\sigma(n,m)$ is contained in the subrepresentation $V$ generated by $\sigma(0,0)$ for all $m\in\ZZ_{\geq0}$. For $n=0$ and $n=2$, the $K$-type $\sigma(n,m)$ has multiplicity one, so by the previous argument it is contained in $V$, and for $n=1$ the $K$-type $\sigma(n,m)$ does not occur in $I(s)$. This shows the claim for $n=0,1,2$, which represents the induction start. For the induction step consider a $K$-type $\sigma(n+3,m+1)$ for $n,m\in\ZZ_{\geq0}$. (Note that $\sigma(n+3,0)$ is a multiplicity-one $K$-type and therefore contained in $V$ by the previous arguments, so we may assume $m\geq0$.) We use Proposition~\ref{prop:v_action} to show that every vector in $(\Gamma_{n+3}\boxtimes\Gamma_{m+1})^L$ can be reached by $R_s(\sigma(n',m'))$ with $0\leq n'<n+3$ and $m'\geq0$. By Lemma~\ref{basemultspace}, when $n+m \equiv 0 \pmod 4$ then $\{v_{n,m}(s,k):k=0,\ldots,[\frac{a(n,m)}{2}]\}$ is a basis for $(\Gamma_n \boxtimes \Gamma_m)^L$, and when $n+m \equiv 2 \pmod 4$, the vectors $\{v_{n,m}'(s,k):k=0,\ldots,[\frac{a(n,m)-1}{2}]\}$ form a basis of $(\Gamma_n \boxtimes \Gamma_m)^L$. We treat the following four cases separately:
\begin{enumerate}[label=(\alph*)]
    \item If $n+m \equiv 0 \pmod 4$ and $a(n,m)$ is even, then $a(n+3,m+1)=a(n,m)+1$ and hence the dimensions of $(\Gamma_n \boxtimes \Gamma_m)^L$ and $(\Gamma_{n+3} \boxtimes \Gamma_{m+1})^L$ agree. By \eqref{prop:v_action1} the space $(\Gamma_{n+3} \boxtimes \Gamma_{m+1})^L$ is reached from $(\Gamma_n\boxtimes\Gamma_m)^L$ if
    $$ (2k + r(n,m) + a(n,m) + s) \neq 0 \qquad \mbox{for all }0 \leq k \leq \left[\frac{a(n,m)}{2}\right], $$
    which is the case since $s \notin \mathbb{Z}$.
    \item If $n+m \equiv 2 \pmod 4$ and $a(n,m)$ is odd, then the same argument as in (a), using \eqref{prop:v_action2} instead of \eqref{prop:v_action1}, shows that $(\Gamma_{n+3} \boxtimes \Gamma_{m+1})^L$ is reached from $(\Gamma_n \boxtimes \Gamma_m)^L$.
    \item If $n+m \equiv 0 \pmod 4$ and $a(n,m)$ is odd, the first $[\frac{a(n,m)}{2}]=\frac{a(n,m)-1}{2}$ vectors in $(\Gamma_{n+3} \boxtimes \Gamma_{m+1})^L$ are reached from $(\Gamma_{n} \boxtimes \Gamma_m)^L$ by the same argument as in (a). We have to show that also the last vector $v_{n+3,m+1}(s,\frac{a(n,m)+1}{2})$ can be reached from lower $K$-types. Assume first that $a(n+3,m+1) = a(n+2,m) + 1 = a(n,m)+1$. Then we also reach the final vector $v_{m+3,n+1}(s,\frac{a(n,m)+1}{2})$ from $(\Gamma_{n+2} \boxtimes \Gamma_m)^L$ using \eqref{prop:v_action4} if
    $$ 0 \neq 3s + r(n+2,m) - a(n+2,m) + m +6\frac{a(n,m)-1}{2} = 3s + \frac{m+n}{2}+m-2. $$
    Since $a(n,m)$ is odd and $n+m\equiv0\pmod2$, a case by case analysis of $n\mod3$ shows that $m$ must be odd, so the above condition is satisfied since $\tilde{s}\not\in2+\frac{2}{3}\ZZ_{\geq0}$, i.e. $s\not\in1+\frac{2}{3}\ZZ_{\leq0}$.
    If not $a(n+3,m+1) = a(n+2,m) + 1$, then we must have $a(n+3,m+1) = a(n,m+2) + 1$, and in this case we can use \eqref{prop:v_action6} to conclude that the last vector $v_{m+3,n+1}(s,\frac{a(n,m)+1}{2})$ is reached from $(\Gamma_{n+2}\boxtimes\Gamma_m)^L$ since $s\not\in\ZZ$.
    \item The case where $n+m\equiv2\pmod4$ and $a(n,m)$ is even is treated similarly as in (c).\qedhere
\end{enumerate}
\end{proof}

For the non-spherical degenerate principal series $\pi_{1,s}$ we find a similar behaviour.

\begin{proposition}
\label{prop:sgnirred}
The representation $\pi_{1,s}$ is reducible if and only if $s$ or $\tilde{s}=3-s$ is contained in $\ZZ_{\geq2}\cup(\frac{7}{3}+\frac{2}{3}\ZZ_{\geq0})$.
\end{proposition}

\begin{proof}
The proof uses the same arguments as the proof of Proposition~\ref{prop:irredtriv}, so we omit it.
\end{proof}

\subsection{Complementary series}
\label{coml-series-sec}

Since both representations $\pi_{\varepsilon,s}$ are irreducible for $s\in(1,2)$, we immediately obtain the corresponding complementary series.

\begin{corollary}\label{cor:ComplementarySeries}
    For both choices of $\varepsilon\in\ZZ/2\ZZ$, the representations $\pi_{\varepsilon,s}$ belong to the complementary series for $s\in(1,2)$, i.e. $\pi_{\varepsilon,s}$ is irreducible and unitarizable for $s\in(1,2)$.
\end{corollary}

\begin{proof}
    This follows from Propositions~\ref{prop:irredtriv} and \ref{prop:sgnirred} and the existence of an intertwining operator $A(\varepsilon,s):I(\varepsilon,s)\to I(\varepsilon,\tilde{s})$ (see also \cite[Proposition 16.4]{Kna86}).
\end{proof}

We now identify the complementary series within the classification of the unitary dual of $G=G_{2(2)}$ by Vogan~\cite{Vog94}.

Let $E'=H_\alpha$, then $\fa_*=\RR E'$ is a maximally split torus in $\fm$. We write $M_\fp A_*N_*\subseteq M$ for the Langlands decomposition of the corresponding minimal parabolic subgroup of $M$ with respect to which $\alpha$ is a positive root. Then $M_\fp A_\fp N_\fp$ is a minimal parabolic subgroup of $G$, where $A_\fp=AA_*$ and $N_\fp=NN_*$. Elements in the dual space of $\fa_\fp=\fa\oplus\fa_*$ are linear combinations of the dual basis $E^*,E'^*$ of $E,E'$. In particular, $\rho_* = E'^*$ and hence $\rho_\fp=\rho+\rho_*=3E^*+E'^*$.

We use induction in stages (see e.g. \cite[Chapter VII]{Kna86})
\begin{flalign*}
& \Ind_{MAN}^G(\Ind_{M_{\mathfrak{p}} A_* N_*}^{M} (\sigma \otimes \exp (\nu_* + \rho_*) \otimes 1) \otimes \exp (\nu + \rho) \otimes 1)
\\ & \cong \Ind_{M_{\mathfrak{p}} A_{\mathfrak{p}} N_{\mathfrak{p}}}^G( \sigma \otimes \exp(\nu_* + \nu + \rho_{\mathfrak{p}}) \otimes 1 ).
\end{flalign*}
Clearly, $1 \subseteq Ind_{M_{\mathfrak{p}} A_* N_*}^M(1 \otimes 1 \otimes 1)$, so for $\nu=2sE^*$ we find
$$ I(0,s) = \Ind_{MAN}^G(1\otimes\exp(\nu)\otimes1) \subseteq \Ind_{M_\fp A_\fp N_\fp}^G(1\otimes\exp(\nu)\otimes1). $$
To identify $I(1,s)$ with a subrepresentation of a principal series, recall from Section \ref{LieAlgPrel} that the subgroup $M_{\mathfrak{p}} \subseteq K$ is identified with
$ \{1,m_\alpha,m_{\alpha+\beta},m_{\alpha+2\beta}\} \cong \mathbb{Z}_2 \times \mathbb{Z}_2 $
and that $1,m_\alpha\in M_0$ and $m_{\alpha+\beta},m_{2\alpha+\beta}\in wM_0$. Therefore, if $\xi$ is the character sending $1,m_{\alpha}$ to $1$ and $m_{\alpha + \beta}, m_{2 \alpha + \beta}$ to $-1$, then $\sgn \subseteq Ind_{M_{\mathfrak{p}} A_* N_*}^{M}(\xi \otimes 1 \otimes 1)$. It follows that for $\nu=2sE^*$
$$ I(1,s) = \Ind_{MAN}^G(\sgn\otimes\exp(\nu)\otimes1) \subseteq \Ind_{M_{\mathfrak{p}} A_{\mathfrak{p}} N_{\mathfrak{p}}}^G(\xi \otimes \exp \nu \otimes 1 ). $$
This implies that $I(0,s)$ resp. $I(1,s)$ is a quotient of $\Ind_{M_{\mathfrak{p}} A_{\mathfrak{p}} N_{\mathfrak{p}}}^G(1 \otimes \exp(2 \rho_{\mathfrak{p}} - \nu) \otimes 1 )$ resp. $\Ind_{M_{\mathfrak{p}} A_{\mathfrak{p}} N_{\mathfrak{p}}}^G(\xi \otimes \exp(2 \rho_{\mathfrak{p}} - \nu) \otimes 1 )$. By the Langlands classification, this principal series representation has a unique irreducible quotient if $\rho_\fp - \nu$ is in the open positive Weyl chamber. We let
$$ \mu = \rho_{\mathfrak{p}} - \nu = (3 - 2s)E^* + (E')^* = (\tfrac{3}{2} - s)(3 \alpha + 2 \beta) +\tfrac{1}{2} \alpha = (5 - 3s) \alpha + (3-2s) \beta $$
and check the conditions for the representations in \cite[Sections 12 and 13]{Vog94}.

We note that for $\mathfrak{g}_2$ the short positive roots are $\alpha, \alpha+\beta$ and $2 \alpha + \beta$ and the long roots are $\beta$, $3 \alpha + \beta$ and $3 \alpha + 2 \beta$. We see that
\begin{align*}
 \langle \mu, H_\alpha \rangle &= 1, & \langle \mu, H_{\alpha+\beta} \rangle &= 4 - 3s, & \langle \mu, H_{2 \alpha+\beta} \rangle &= 5 - 3 s,\\
 \qquad \langle \mu, H_\beta \rangle &= 1-s, & \langle \mu, H_{3 \alpha + \beta} \rangle  &= 2-s, & \langle \mu, H_{3 \alpha + 2 \beta} \rangle  &= 3-2s. 
\end{align*}
This implies that $\mu$ is in the open positive Weyl chamber for $1 < s < \frac{3}{2}$, and for $\frac{3}{2} < s < 2$ there is an equivalence $I(\varepsilon,s)\cong I(\varepsilon,3-s)$, so it suffices to consider $1 < s < \frac{3}{2}$.
Then, for $\Ind_{M_{\frakp} A_{\frakp} N_{\frakp}}^G(1 \otimes \exp(\mu) \otimes 1)$ we are in the case of \cite[Theorem 12.1 c) i), ii)]{Vog94}. In the case of $\Ind_{M_{\frakp} A_{\frakp} N_{\frakp}}^G(\xi \otimes \exp(\mu) \otimes 1)$ we have a representation $\xi$ of $M_{\mathfrak{p}}$ with $\xi(m_{\alpha}) = \xi(m_{3 \alpha + 2 \beta}) = 1$. Hence we are in the case \cite[Theorem 13.2 d) i)]{Vog94}, with $\alpha$ and $3 \alpha + 2 \beta$ being the short and long roots.

\begin{figure}[H]
\caption{The complementary series (grey) for $\Ind_{M_{\frakp} A_{\frakp} N_{\frakp}}^G(1 \otimes \exp(\mu) \otimes 1)$ (left) and $\Ind_{M_{\frakp} A_{\frakp} N_{\frakp}}^G(\xi \otimes \exp(\mu) \otimes 1)$ (right) and the parameters for which the complementary series of $I(\varepsilon,s)$ occurs as Langlands quotient (red)}
\label{fig-no-action}
\centering
\includegraphics[width=0.45\textwidth]{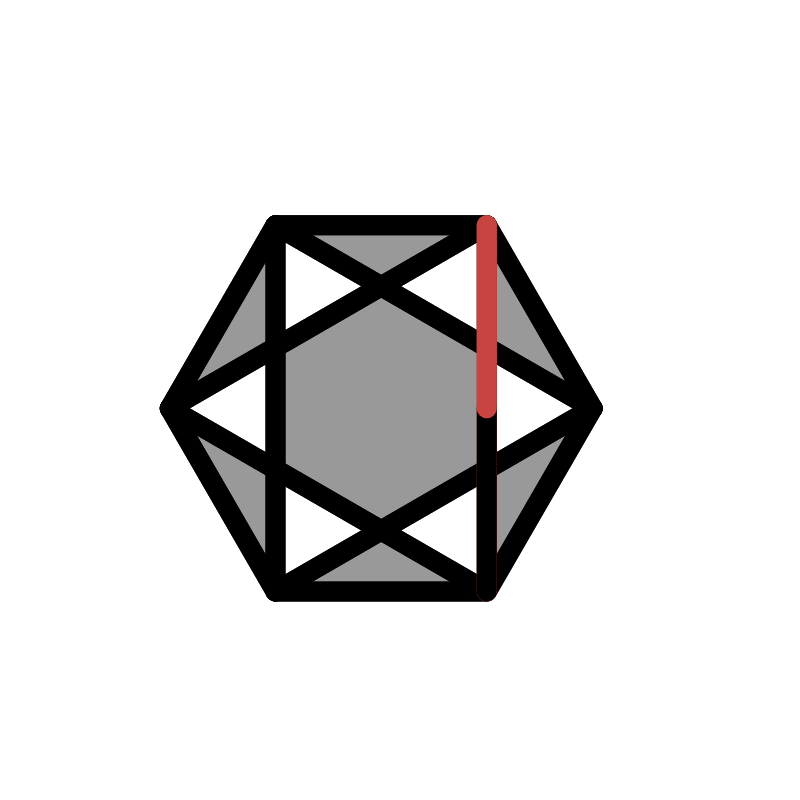}
\includegraphics[width=0.45\textwidth]{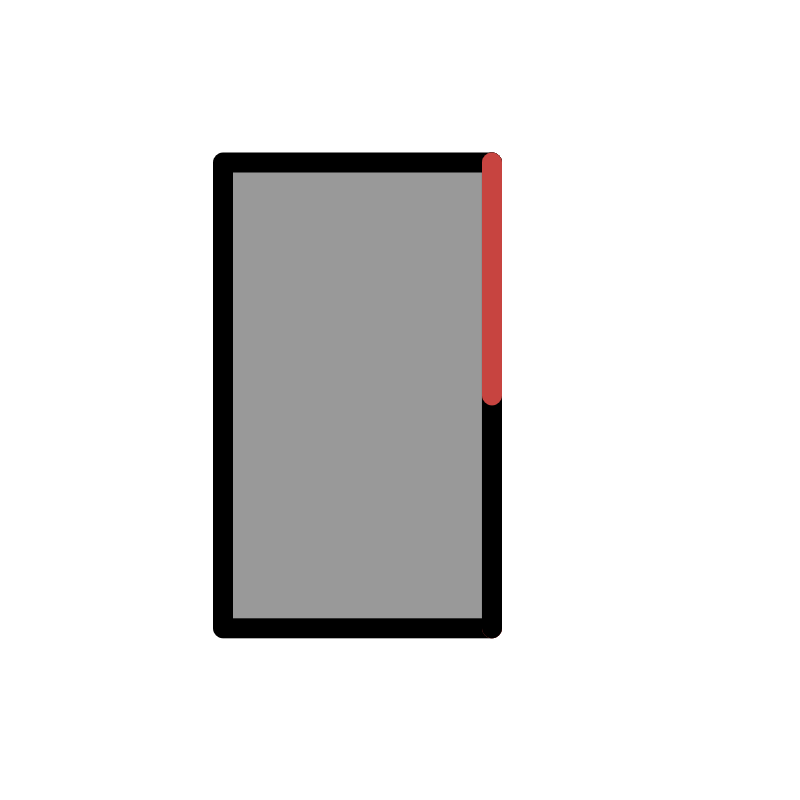}
\end{figure}
Figure~\ref{fig-no-action} shows the complementary series representations of $\Ind_{M_{\frakp} A_{\frakp} N_{\frakp}}^G(1 \otimes \exp(\mu) \otimes 1)$ and $\Ind_{M_{\frakp} A_{\frakp} N_{\frakp}}^G(\xi \otimes \exp(\mu) \otimes 1)$ in grey and the parameters for which the complementary series of $I(0,s)$ resp. $I(1,s)$ occur as irreducible quotient (Langlands quotient) in red. The figures show the two-dimensional space of roots for $\mathfrak{g}_2$ and have the short root $\alpha$ as the $x$-coordinate and the long root $3 \alpha + 2 \beta$ as the $y$-coordinate. See \cite{Vog94} for a more detailed description.

\begin{figure}[H]
\caption{The complementary series (grey) for $\Ind_{M_{\frakp} A_{\frakp} N_{\frakp}}^G(1 \otimes \exp(\mu) \otimes 1)$ (left) and $\Ind_{M_{\frakp} A_{\frakp} N_{\frakp}}^G(\xi \otimes \exp(\mu) \otimes 1)$ (right) and the parameters for which the complementary series of $I(\varepsilon,s)$ occurs as irreducible subquotient (red)}
\label{fig-compl-series-full}
\centering
\includegraphics[width=0.45\textwidth]{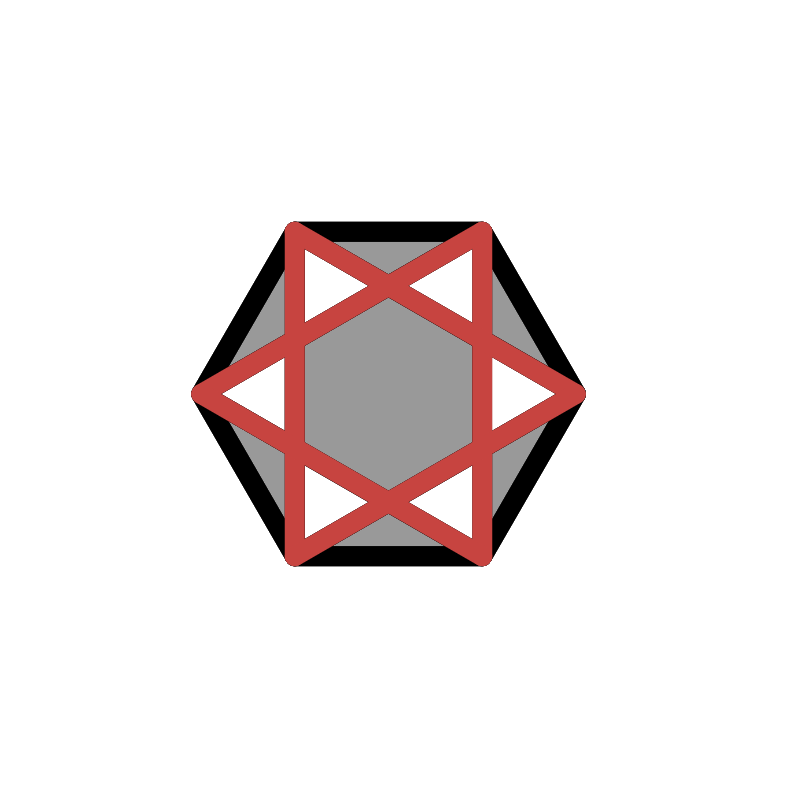}\includegraphics[width=0.45\textwidth]{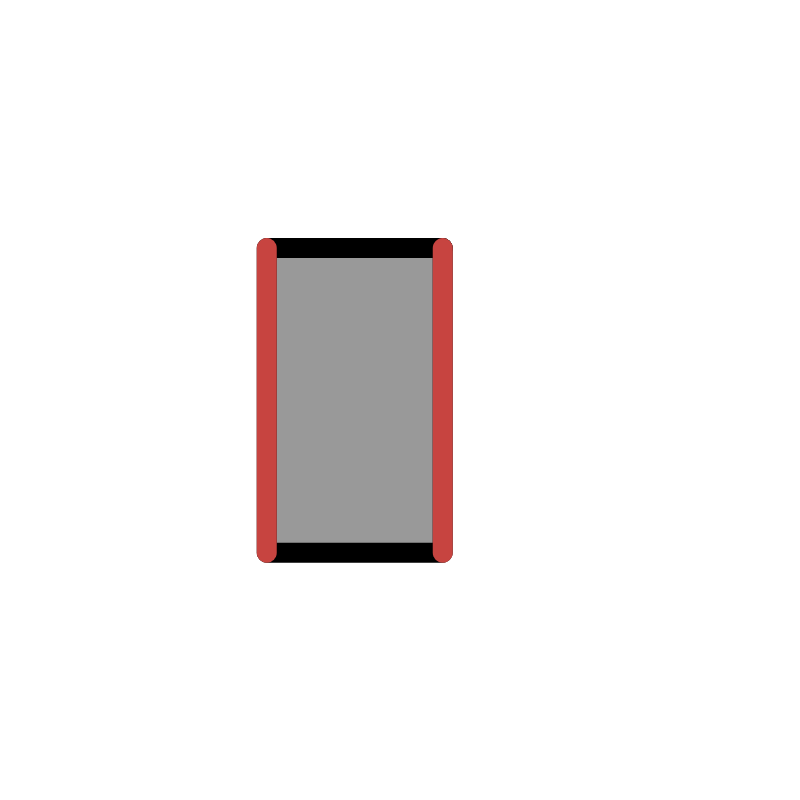}
\end{figure}
By applying the action of the Weyl group $N_K(\mathfrak{a}_{\mathfrak{p}})/M_{\mathfrak{p}}$, we can further identify the complementary series representations of $I(\varepsilon,s)$ as subquotients of the principal series for a large part of the boundary of the complementary series region for the principal series, see Figure~\ref{fig-compl-series-full}.

\subsection{Some irreducible subrepresentations}\label{sec:MinrepAndDoubleLadder}

It follows from \cite[Theorem 4.8]{Kab12} that the minimal representation (also called ladder representation) of $G$ is a subrepresentation of $I(1,s)$ for $s = \frac{2}{3}$ with $K$-types given by the embedding vectors $v_{3m+2,m}(\frac{2}{3},0)$ ($m\in\ZZ_{\geq0}$). Furthermore, it follows from \cite[Theorem 4.10]{Kab12} that there is an irreducible subrepresentation (also called double ladder representation) inside $I(0,s)$ for $s = \frac{1}{3}$ with $K$-types given by the embedding vectors $v_{3m,m}(\frac{1}{3},0)$ and $v_{3m+4,m}(\frac{1}{3},0)$ ($m\in\ZZ_{\geq0}$). We show that these two subrepresentations are precisely the kernel of the corresponding intertwining operator.

For this, we normalize the meromorphic family of intertwining operators $A(\varepsilon,s)$ so that
$$ A_{0,0}(s)\equiv1 \quad \mbox{for }\varepsilon=0 \qquad \mbox{and} \qquad A_{2,0}(s)\equiv1 \quad\mbox{for }\varepsilon=1. $$

\begin{proposition}
\label{prop:ladderker}
The kernel of the intertwining operator $(s-\frac{2}{3})A(1,s):I(1,s)\to I(1,\tilde{s})$ at $s = \frac{2}{3}$ is the ladder representation of $G$. It is also the image of the intertwiner $A(1,\tilde{s}):I(1,\tilde{s})\to I(1,s)$. Moreover, this subrepresentation has $K$-types $\sigma(3m+2,m)$ ($m\in\ZZ_{\geq0}$) with corresponding embedding vectors $v_{3m+2,m}(s,0)$.
\end{proposition}

\begin{proof}
    First note that $A(1,\tilde{s})$ is regular at $\tilde{s}=\frac{7}{3}$ by the formulas in Corollary~\ref{cor:ClosedFormulasIntertwinerMultOneKtypes} and Theorem~\ref{thm:ClosedFormulasIntertwinerEigenvalues}. A close look at the formulas in Corollary~\ref{cor:ClosedFormulasIntertwinerMultOneKtypes} shows that $A_{n,m}(\tilde{s})=0$ for all multiplicity-one $K$-types $\sigma(n,m)$ inside $I(1,\tilde{s})$ except $A_{2,0}(\tilde{s})$, which is normalized to be one. Inspecting the formulas in Theorem~\ref{thm:ClosedFormulasIntertwinerEigenvalues} thoroughly reveals that all $\mu_k(n,m,\tilde{s})=0$ except for $\mu_0(3m+2,m,\tilde{s})$ which turns out to be
    $$ \mu_0(3m+2,m,\tilde{s}) = \frac{(s+1)_m}{(\tilde{s}+1)_m}. $$
    This shows that the image of $A(1,\tilde{s})$ is a subrepresentation of $I(1,s)$ with $K$-types $\sigma(3m+2,m)$ ($m\in\ZZ_{\geq0}$), the ladder representation.

    A similar analysis of $A(1,s)$ shows that $A_{n,m}(s)$ in Corollary~\ref{cor:ClosedFormulasIntertwinerMultOneKtypes} has a simple pole at $s=\frac{2}{3}$ for all multiplicity-one $K$-types $\sigma(n,m)$ except $A_{2,0}(s)$. Consequently, all eigenvalues $\mu_k(n,m,s)$ in Theorem~\ref{thm:ClosedFormulasIntertwinerEigenvalues} have a simple pole at $s=\frac{2}{3}$ except $\mu_0(3m+2,m,s)$, which is regular at $s=\frac{2}{3}$. We can therefore regularize $\tilde{A}(s)=(s-\frac{2}{3})A(s)$ to obtain an intertwiner which is regular at $s=\frac{2}{3}$ and whose eigenvalues are all non-zero except $\mu_0(3m+2,m,s)$. Its kernel is again the ladder representation.
\end{proof}

We can prove a similar statement for the double ladder representation.

\begin{proposition}
\label{prop:doubleladderker}
The kernel of the intertwining operator $(s-\frac{1}{3})A(0,s):I(0,s)\to I(0,\tilde{s})$ at $s = \frac{1}{3}$ is the double ladder representation of $G$. It is also the image of the intertwiner $A(0,\tilde{s}):I(0,\tilde{s})\to I(0,s)$. Moreover, this subrepresentation has $K$-types $\sigma(3m,m)$ and $\sigma(3m+4,m)$ ($m\in\ZZ_{\geq0}$) with corresponding embedding vectors $v_{3m,m}(s,0)$ and $v_{3m+4,m}(s,0)$.
\end{proposition}

\begin{proof}
The same arguments as in the proof of Proposition~\ref{prop:ladderker} show that only $\mu_0(3m,m,s)$ and $\mu_0(3m+4,m,s)$ are regular at $s=\frac{1}{3}$ while all other $\mu_k(n,m,s)$ have a single pole. Hence, $(s-\frac{1}{3})\mu_k(n,m,s)=0$ if and only if $(n,m)=(3m,m)$ or $(3m+4,m)$ and $k=0$. So the kernel of $(s-\frac{1}{3})A(0,s)$ at $s=\frac{1}{3}$ is the double ladder representation. Similar considerations apply to $A(0,\tilde{s})$.
\end{proof}

\subsection{Quaternionic discrete series and limit of discrete series}

In \cite[Propositions 5.7 and 8.4]{GW96} a series $\pi_k^{\textup{GW}}$, $k\geq2$, of irreducible unitary representations of $G$ is constructed using the cohomology of certain line bundles
over a $\mathbb P^1$-bundle over $G/K$. The lowest $K$-type of $\pi_k^{\textup{GW}}$ is $\sigma(0,k-2)$. For $k\geq5$ it is a quaternionic discrete series representation, and for $k=4$ it is a limit of discrete series. We show that for $k\geq4$ even, the representation $\pi_k^{\textup{GW}}$ occurs as a subrepresentation of $\pi_{\varepsilon,s}$ for $s=\frac{k}{2}$ and $\varepsilon\equiv\frac{k}{2}+1\pmod2$, and identify it in terms of the intertwining operators $A(\varepsilon,s)$.

Note that the infinitesimal character of $\pi_k^{\textup{GW}}$ is $\rho - \frac{k}{2}(3 \alpha + 2 \beta)$ by \cite[Proposition 5.7]{GW96}, while the infinitesimal character of $\pi_{\varepsilon,s}$ is $\rho - s (3 \alpha + 2 \beta)$ (follows e.g. from \cite[Proposition 8.22]{Kna86}). Since $\pi_{\varepsilon,s}$ is irreducible for $s=\frac{k}{2}$ with $k$ odd, the quaternionic discrete series $\pi_k^{\textup{GW}}$ cannot occur as a subrepresentation of $\pi_{\varepsilon,s}$ due to the two representations having different $K$-multiplicities.\\

We first show that $\pi_k^{\textup{GW}}$ is indeed a subrepresentation of $\pi_{\varepsilon,s}$.

\begin{proposition}\label{prop:QDS}
Let $k\geq4$ be even and put $s=\frac{k}{2}$ and $\varepsilon\equiv\frac{k}{2}+1\pmod2$. Then $\pi_k^{\textup{GW}}$ is a subrepresentation of $I(\varepsilon,s)$.
\end{proposition}

\begin{proof}
For $k\geq6$ the representation $\pi_k^{\textup{GW}}$ belongs to the discrete series, so the claim follows from \cite[Theorem 2.3]{Yoshi98} (note the typo, $\tilde H_2$ should be $\tilde H_1$), where it was stated without a proof. See also \cite{GGS02}.

Now we prove that the limit of the quaternionic discrete series $\pi_4^{\textup{GW}}$ is a subrepresentation of $I(1,2)$. Write $\calH_k$ for the $(\frakg,K)$-module of $\pi_k^{\textup{GW}}$. We already know that $\calH_6$ can be embedded into $I(0,3)$.
Choose the positive system $\Delta^+$ relative to the infinitesimal character $\rho - 3(3 \alpha + 2 \beta)$, which is non-singular. Note that this is also a choice of positive system for $\rho - 2(3 \alpha + 2 \beta)$ (which is singular) and that the weight $- (3 \alpha + 2 \beta)$ is dominant integral with respect to $\Delta^+$. Further note that the adjoint representation $\frakg_2$ has lowest weight $3 \alpha + 2 \beta$. Thus, applying the Zuckerman translation functor we obtain $\mathcal{H}_4 \cong \psi_{\rho - 2(3 \alpha + 2 \beta)}^{\rho - 3(3 \alpha + 2 \beta)} \mathcal{H}_6$, so that $\mathcal{H}_4 \cong p_{\rho - 2(3 \alpha + 2 \beta)} (\mathcal{H}_6 \otimes \frakg_2)$, where $p_{\rho - 2(3 \alpha + 2 \beta)}$ is the projection onto the space with (generalized) infinitesimal character $\rho - 2(3 \alpha + 2 \beta)$.

Recall now that pointwise multiplication defines for all $\varepsilon,\varepsilon'\in\ZZ/2\ZZ$, $s,s'\in\CC$,
a $\fg_2$-equivariant map
$$ I(\varepsilon, s)
\otimes
I(\varepsilon', s')
\to I(\varepsilon+
\varepsilon', s+s'), \quad
f_1\otimes f_2\mapsto [g\mapsto f_1(g)f_2(g), g\in G].
$$
Moreover, the Lie algebra
$\frakg_2$ is realized
as a subrepresentation in
$I(1,-1)$ via
the matrix coefficients,
$v \mapsto [g \mapsto \langle g^{-1} v, v_{-3\alpha - 2 \beta} \rangle]$, where $v_{-3\alpha - 2 \beta}\in\frakg_2$ is a vector of weight $-3\alpha-2\beta$ with respect to $\frakh$.
We obtain a
$\fg_2$-equivariant map
 $\mathcal{H}_6 \otimes \frakg_2 \rightarrow I(1,2)$ by  composing the two maps
$$
I(0,3) \otimes \mathfrak{g}_2 \rightarrow I(1,2),\quad
f \otimes v \mapsto \left[ g \mapsto \langle g^{-1} v, v_{-3 \alpha - 2 \beta} \rangle f(g)\right].$$
The image of $\mathcal{H}_6 \otimes \frakg_2$ has infinitesimal character $\rho - 2(3\alpha + 2 \beta)$, hence it is a quotient of the limit of the discrete series. It is also non-zero, so it is isomorphic to $\calH_4$.
\end{proof}

We claim that the limit of the quaternionic discrete series $\pi_4^{\textup{GW}}$ is isomorphic to the kernel of the Knapp--Stein operator $A(1,2): I(1, 2)\to I(1, 1)$. Again, we normalize $A(1,s)$ such that $A_{2,0}(s)\equiv1$.

\begin{theorem}\label{thm:LDS}
The kernel 
$\ker(A(1,2))$
is
isomorphic to the limit of discrete series $\pi_4^{\textup{GW}}$.
\end{theorem}

\begin{proof}
By Proposition~\ref{prop:QDS}, both $\ker(A(1,2))$ and $\pi_4^{\textup{GW}}$ occur as subrepresentation of $I(1,2)$. Since $\pi_4^{\textup{GW}}$ is irreducible and contains the $K$-type $(0,2)$ which occurs in $I(1,2)$ with multiplicity one, it suffices to show that $\ker(A(1,2))$ is generated by this $K$-type.

We first observe that by Corollary \ref{cor:ClosedFormulasIntertwinerMultOneKtypes} the multiplicity-one $K$-types that are contained in $\ker(A(1,2))$ are $(0,4r+2)$ with $r \geq 0$, $(2,4r)$ with $r \geq 1$, and $(4,4r+2)$ with $r \geq 1$. It is easy to see using the formulas in Lemma \ref{LieAlg-one-mult-K-types} that these $K$-types are in fact generated by the $K$-type $(0,2)$. We write $B\subseteq\{\sigma\in \hat K: 
\sigma\subseteq I(1,2)|_K \}$ for the subset of multiplicity-one $K$-types in  $\ker A(1,2)$.

Next, a careful analysis of the eigenvalues in Theorem \ref{thm:ClosedFormulasIntertwinerEigenvalues} shows that the $K$-types that appear in $\ker(A(1,2))$ are $(3m,m+2r)$ with $r \geq 1$, $(3m+2,m+2r)$ with $r \geq 2$, and $(3m+4,m+2r)$ with $r \geq 3$. We have to prove that they are generated by the $K$-type $(0, 2)$. 

Note that for generic $s$, the whole multiplicity spaces of these $K$-types are generated from the multiplicity-one $K$-types in $B$ by applying the maps $T_{n,m}^{n+3,m+1}(s)$ and $T_{n,m}^{n+3,m-1}(s)$ (see the proofs of Proposition~\ref{evprop} and Theorem~\ref{thm:ClosedFormulasIntertwinerEigenvalues}). Here, in every step, we need $T_{n,m}^{n+3,m+1}(s)$ to produce the first vectors of the relevant $K$-type, and possibly $T_{n,m}^{n+3,m-1}(s)$ for the last one. For instance, applying $T_{3,3}^{6,4}(s)\circ T_{0,2}^{3,3}(s)$ to the multiplicity-one $K$-type $(0,2)$ and $T_{3,5}^{6,4}(s)\circ T_{0,6}^{3,5}(s)$ to the multiplicity-one $K$-type $(0,6)$ generically yields a basis of the multiplicity-two $K$-type $(6,4)$.
In summary, we can find vectors $v_i$ in the multiplicity-one $K$-types in $B$ as well as linear maps $T_i(s)$ which are concatenations of the different $T_{n',m'}^{n'+3,m'+1}(s)$ and $T_{n',m'}^{n'+3,m'-1}(s)$ such that $(T_i(s)v_i)_i$ is a basis of the multiplicity space of $(n,m)$ for generic $s$, in particular around $s=2$ (but in general not for $s=2$ as some of the coefficients in $T_{n',m'}^{n'+3,m'-1}(s)$ will vanish here).

We show that $(T_i(s)v_i)_i$ is a basis of the multiplicity space of $(n,m)$ also for $s=1$. Clearly, the coefficients appearing in Proposition~\ref{prop:v_action} for $T_{n',m'}^{n'+3,m'+1}(s)$ are non-zero for $s=1$, because $2k+r(n',m')+a(n',m')+s\geq s>0$. We will do an analysis of the coefficients modulo $4$ appearing in the relevant $T_{n',m'}^{n'+3,m'-1}(s)$ to show that they also do not vanish at $s=1$. Note that we only need $T_{n',m'}^{n'+3,m'-1}(s)$ when $(n'+3) + (m'-1) \equiv 0 \mod 4$ and $a(n'+3,m'-1) = a(n',m') + 1$ is even, or $(n'+3) + (m'-1) \equiv 2 \mod 4$ and $a(n'+3,m'-1) = a(n',m') + 1$ is odd. Consider the first case, then $n'+m' \equiv 2 \mod 4$ and we want to obtain $v_{n'+3,m'-1}(s,k+1)$ with $s=1$ and $k=\frac{a(n',m')-1}{2}$. We see that, as $a(n',m')$ is odd, both $n'$ and $m'$ are odd. Now the coefficients of $v_{n'+3,m'-1}'(s,k+1)$ in $T_{n',m'}^{n'+3,m'-1}(s)v_{n',m'}(s,k)$ consist of (among others)
$$ 2s + n' - m' - 2 a(n',m') + 4k = n'-m'-2a(n',m') \equiv 2 \mod 4 $$
if $s=1$ and $k=\frac{a(n',m')-1}{2}$, hence they are non-zero. A similar analysis holds for the rest of the coefficients. It follows that $(T_i(s) v_i)_i$ is a basis of the multiplicity space of $(n,m)$ for $s=1$.

Now let $(n,m)$ be one of the $K$-types in $\ker(A(1,2))$, let $v$ be a vector in the multiplicity space of $(n,m)$ and write $v = \sum_i f_i(s) T_i(s) v_i$, where $f_i$ is a rational function in $s$, possibly singular at $s=2$. Then
$A(1, s): I(1, s)
\to I(1, \tilde s)$
and
$$ 
A(1,s) v = \sum_i f_i(s) T_i(\tilde{s}) A(1,s) v_i.$$
As $A(1,s) v_i$ is a constant multiple of $v_i$, with a zero at $s=2$ of multiplicity exactly one, and $(T_i(1) v_i)_i$ is a basis of $(n,m)$, we see that for $v \in \ker(A(1,2))$ the $f_i(s)$ cannot be singular at $s=2$. We conclude that $v = \sum_i f_i(2) T_i(2) v_i$ is actually generated by applying the Lie algebra actions $T_i(2)$ to the vectors $v_i$ in the multiplicity-one $K$-types in $B$, which are generated by the multiplicity-one $K$-type $(0,2)$ as shown above. Hence, $v$ is also contained in the subrepresentation generated by $(0,2)$. This finishes the proof.
\end{proof}

For $k\geq6$ even, it turns out that $\pi_k^{\textup{GW}}$ is a proper subrepresentation of the kernel of the intertwining operator $A(\varepsilon,s)$, where $\varepsilon \equiv s+1 \mod 2$. Here we normalize $A_{0,0}(s)\equiv1$ and $A_{2,0}(s)\equiv1$ as in Section~\ref{sec:MinrepAndDoubleLadder}. To describe this subrepresentation precisely, we need to recall the Jantzen filtration of $I(\varepsilon,s)$ (see e.g. \cite[Definition 3.7]{Vog84Unit} or \cite[Definition 7.4]{Vog94}). For this note that all induced representations $I(\varepsilon,s)$ for fixed $\varepsilon$ are realized on the same space $I(\varepsilon)=\{f\in L^2(K/L_0):f(gw)=(-1)^\varepsilon f(g)\}$.

\begin{definition}
The $n$-th level $I(\varepsilon,s)_n$ of the Jantzen filtration of $I(\varepsilon,s)$ consists of those vectors $v \in I(\varepsilon)$ such that there is an analytic function $v(z)$ around $z = s$ which has the property that $v(s) = v$ and $A(\varepsilon,z) v(z)$ vanishes to degree at least $n$ at $z = s$.
\end{definition}

An alternative way of stating Theorem~\ref{thm:LDS} is that the limit of discrete series $\pi_4^{\textup{GW}}$ is isomorphic to the first level of the Jantzen filtration $I(1,2)_1$. We now prove that the quaternionic discrete series $\pi_k^{\textup{GW}}$, $k\geq6$ even, is isomorphic to the second level of the Jantzen filtration $I(\varepsilon,s)_2$, where $s=\frac{k}{2}$ and $\varepsilon\equiv \frac{k}{2}+1\pmod2$. The strategy is the same as in the proof of Theorem~\ref{thm:LDS}, but
now the Jantzen filtration of $I(\varepsilon, s)$ for $s\ge 3$ is longer than that for $s=2$, so additional considerations are necessary.

Fix $k\geq6$ and let $s=\frac{k}{2}$, $\varepsilon=\frac{k}{2}+1$. We know that the $K$-types in $I(\varepsilon,s)$ are of the form $(3m,m+2r)$, $(3m+2,m+2r)$, $(3m+4,m+2r)$ and $(3m+2r,m)$ with $r,m\in\ZZ_{\geq0}$.
\begin{enumerate}[label=(\Alph*)]
    \item\label{CaseA} The $K$-types $(3m+j,m+2r)$ ($j=0,2,4$) are generically obtained from the multiplicity-one $K$-types
    $B = \{ (j,a) \}_{2r \leq a \leq 2(r+m), a \ \mathrm{even}, \ \frac{a}{2} \equiv s \mod 2 }$. Note that the multiplicity-one $K$-type $(j,a)$ only appears in $I(\varepsilon,s)$ when $j+a \equiv 2 (s+1) \mod 4$. Here we only need the maps $T_{n',m'}^{n'+3,m'+1}(z)$, and for the last basis vector also $T_{n',m'}^{n'+3,m'-1}(z)$. As in the proof of Theorem~\ref{thm:LDS}, we choose for a $K$-type $(3m+j,m+2r)$ vectors $v_i$ in the multiplicity-one $K$-types in $B$ as well as linear maps $T_i(z)$ which are concatenations of maps of the form $T_{n',m'}^{n'+3,m'+1}(z)$ and $T_{n',m'}^{n'+3,m'+1}(z)$, such that $(T_i(z)v_i)_i$ is a basis of the multiplicity space of $(3m+j,m+2r)$ for generic $z\in\CC$.
    \item\label{CaseB} For the $K$-types $(3m+2r,m)$, we only need the maps $T_{n',m'}^{n'+3,m'+1}(z)$ and $T_{n',m'}^{n'+1,m'+1}$ and the multiplicity-one $K$-types $C = \{ (a,0) \}_{2r \leq a \leq 2(r+m), a \ \mathrm{even}, \ \frac{a}{2} \equiv s \mod 2 }$. In a similar way, we let $(T_i(z)v_i)_i$ be a basis of the multiplicity space of $(3m+2r,m)$ for generic $z$.
\end{enumerate}

Let us first analyze the $K$-types from (A).

\begin{lemma}
\label{CountZero}
For the $K$-types $(3m,m+2r)$, $(3m+2,m+2r)$ and $(3m+4,m+2r)$, the vectors $T_i(\tilde{z})v_i$ defined in \ref{CaseA} have a zero at $z=s$ of multiplicity at most one, and this can only happen when $r \leq s-3$ for $(3m,m+2r)$, $r \leq s-4$ for $(3m+2,m+2r)$ and $r \leq s-5$ for $(3m+4,m+2r)$.
\end{lemma}

\begin{proof}
We prove everything for $s$ odd, the statements for $s$ even are proven similarly. First we see that $T_{n',m'}^{n'+3,m'-1}(\tilde{z})$ is only needed in $T_i(\tilde z)$ for a step when $(n'+3) + (m'-1) \equiv 0 \mod 4$ and $a(n'+3,m'-1) = a(n',m') + 1$ is even, or when $(n'+3) + (m'-1) \equiv 2 \mod 4$ and $a(n'+3,m'-1) = a(n',m') + 1$ is odd. We then see that the coefficients for $T_{n',m'}^{n'+3,m'-1}(\tilde{s})$ are not zero by the following argument. First consider $(n'+3) + (m'-1) \equiv 0 \mod 4$ and $a(n'+3,m'-1) = a(n',m') + 1$ is even, then we need to create $v_{n'+3,m'-1}(\tilde{s},\frac{a(n',m')+1}{2})$. We see then $m'$ is odd and we see that $2 \tilde{s} + n' - m' - 2 a(n',m') + 2 a(n',m') + 2 \equiv 2 \mod 4$. A similar analysis holds for the other coefficients and the case $(n'+3) + (m'-1) \equiv 2 \mod 4$ and $a(n'+3,m'-1) = a(n',m') + 1$ is odd.

Now we look at zeros which are created by $T_{n',m'}^{n'+3,m'+1}(\tilde{z})$ at $z = s$. We see that we only get possibly one zero for the coefficient at $v_{n',m'}(\tilde{z},k)$. This only appears for $(3m,m+2r)$ when $r \leq s-3$, $(3m+2,m+2r)$ when $r \leq s-5$ and $(3m+2,m+2r)$ when $r \leq s-5$.
\end{proof}

Now we prove a similar result for the $K$-types in (B).

\begin{lemma}
\label{CountZero2}
For the $K$-types $(3m+2r,m)$, the vectors $T_i(\tilde{z})v_i$ defined in \ref{CaseB} have a zero at $z=s$ of multiplicity at most one, and this can only happen when $r \leq s-3$.
\end{lemma}

\begin{proof}
Let $s$ be odd, the $s$ is even case is done by a similar analysis. When $r \geq s-1$ we immediately see that the coefficients of the appearing $T_{n',m'}^{n'+3,m'+1}(\tilde{s})$ are non-zero. The same is true for the coefficients in $T_{n',m'}^{n'+1,m'+1}(\tilde{s})$ by the following argument. We see that for $(n'+1,m'+1)$ we need $T_{n',m'}^{n'+1,m'+1}$ when $(n'+1)+(m'+1) \equiv 0 \mod 4$ and $a(n'+1,m'+1) = a(n',m')+1$ is even, and when $(n'+1)+(m'+1) \equiv 2 \mod 4$ and $a(n'+1,m'+1) = a(n',m')+1$ is odd. We study $(n'+1)+(m'+1) \equiv 0 \mod 4$ and $a(n'+1,m'+1) = a(n',m')+1$ is even, then we see that we need $T_{n',m'}^{n'+1,m'+1}(\tilde{s}) v_{n',m'}'(\tilde{s},\frac{a(n',m')-1}{2})$. We see that $3 \tilde{s} + r(n',m')-a(n',m')+m'+6 \frac{a(n',m')-1}{2} = 3 \tilde{s} + \frac{n'+m'}{2} + m' - 3 \equiv 1 \mod 2$, so non-zero. The other cases work in a similar way.

Now we study the case $r \leq s-3$ at $z = s$. Since the Lie algebra action is affine linear in $s$, the vectors $T_i(\tilde{z}) v_i$ depend polynomially on $z$. If there is a zero at $z=s$, this either comes from having applied a map of the form $T_{n',m'}^{n'+3,m'+1}(\tilde{z})$ or $T_{n',m'}^{n'+1,m'+1}(\tilde{s})$. If it comes from $T_{n',m'}^{n'+3,m'+1}(\tilde{z})$, it must be that some $(2k + r(n',m') + a(n',m') + \tilde{z}) v_{n'+3,m'+1}(\tilde{z},k)$ gives a zero at $z = s$ as a coefficient of $v_{n'+3,m'+1}(\tilde{z},k)$. Applying more transformations of the form $T_{n'',m''}^{n''+3,m''+1}(\tilde{z})$ and $T_{n'',m''}^{n''+1,m''+1}(\tilde{z})$, we cannot get more zeros as coefficients of $v_{n'',m''}(\tilde{z},k)$ at $z = s$, which can be seen by comparing these coefficients to $(2k + r(n',m') + a(n',m') + \tilde{z})$ from before.

A zero can also arise from applying $T_{n',m'}^{n'+1,m'+1}(\tilde{z})$. By the above argument and by an analysis of the coefficients modulo $2$, we must have that $n'+m' \equiv 2 \mod 4$, $a(n',m')$ odd and $(\frac{n'+m'}{2} + \tilde{s} - 1) = 0$, so we get $(\frac{n'+m'}{2} + \tilde{z} - 1) p(z)$ as the coefficient for $v_{n'+1,m'+1}(\tilde{z}, \frac{a(n',m')+1}{2})$, where $p(z)$ has no zero at $z = s$. By a similar analysis as before, we do not get higher multiplicities of our zeros.
\end{proof}

\begin{corollary}
\label{zerocoro}
The elements $T_i(z)v_i$ forming a basis of the multiplicity space of a $K$-type $(n,m)$ have zeros of degree at most $1$, and this can only happen when $v_i$ in the multiplicity-one subspace is not in the kernel of $A(\varepsilon,s)$.
\end{corollary}

\begin{proof}
This follows directly from Lemma~\ref{CountZero}, Lemma~\ref{CountZero2} and Corollary~\ref{cor:ClosedFormulasIntertwinerMultOneKtypes} by carefully comparing the indices.
\end{proof}

\begin{theorem}\label{thm:QDS}
Let $k\geq6$ be even and let $s=\frac{k}{2}$, $\varepsilon\equiv\frac{k}{2}+1\pmod2$. Then the second level $I(\varepsilon,s)_2$ of the Jantzen filtration of $I(\varepsilon,s)$ is isomorphic to the quaternionic discrete series representation $\pi_k^{\textup{GW}}$.
\end{theorem}

\begin{proof}
Since $I(\varepsilon,s)_2$ and $\pi_k^{\textup{GW}}$ are both subrepresentations of $I(\varepsilon,s)$ and share the multiplicity-one $K$-type $(0,k-2)$, it suffices to show that this $K$-type generates $I(\varepsilon,s)_2$. So let $v\in I(\varepsilon,s)_2$ in the multiplicity space of $(n,m)$, then there exists an analytic function $v(z)$ defined around $z=s$ in $(n,m)$ such that $v(z) = v$ and $A(\varepsilon, z) v(z)$ has a zero of order $2$. We can write $v(z) = \sum_i f_i(z) T_i(z) v_i$, where $f_i$ is a rational function in $z$, because the $T_i(z) v_i$ are generically a basis. Then
$$ A(\varepsilon,z) v(z) = \sum_i f_i(z) T_i(\tilde{z}) A(\varepsilon, z) v_i.$$ 
Note that $A(\varepsilon,z) v_i$ is a constant multiple of $v_i$ for every $i$. Since the left hand side vanishes of order two at $z=s$ and the vectors $T_i(\tilde{z})v_i$ generically form a basis, every term $f_i(z)T_i(\tilde{z})A(\varepsilon,z)v_i$ has to vanish of order two at $z=s$. So if some $f_i(z)$ does not vanish at $z=s$, the term $T_i(\tilde{z})A(\varepsilon,z)v_i$ has to vanish of order at least two. But by Corollary~\ref{zerocoro} this can only happen if $v_i$ is contained in $I(\varepsilon,s)_2$, i.e. $v_i$ is in $(0,2r)$ with $r \geq s-1$, $(2,2r)$ with $r \geq s$, or $(4,2r)$ with $r \geq s+1$ (see Corollary~\ref{cor:ClosedFormulasIntertwinerMultOneKtypes}). Moreover, $f_i(z)$ cannot have a pole at $z=s$, otherwise $T_i(\tilde{z})A(\varepsilon,z)v_i$ would have a zero of order at least three, and since $A(\varepsilon,z)v_i$ vanishes of order at most two, $T_i(\tilde{z})v_i$ would have to vanish as well, which is impossible by Corollary~\ref{zerocoro}. It follows that $v=v(s)=\sum_i f_i(s)T_i(s)v_i$ is generated by the $v_i$ which are contained in the multiplicity-one $K$-types in $I(\varepsilon,s)_2$. As in the proof of Theorem~\ref{thm:LDS} it is easy to see using Lemma~\ref{LieAlg-one-mult-K-types} that these multiplicity-one $K$-types are in fact generated by the $K$-type $(0,k-2)$, so this $K$-type generates $I(\varepsilon,s)_2$.
\end{proof}

\begin{remark}
For the sake of completeness, we would like to point out that at the integer points $s\in\ZZ_{\geq3}$, $\varepsilon \equiv s + 1\pmod 2$, there is a sequence of subrepresentations
$$ \{0\}\subseteq I(\varepsilon,s)_2 \subseteq I(\varepsilon,s)_1\subseteq I(\varepsilon,s).$$
Here $I(\varepsilon,s) / I(\varepsilon,s)_1 \cong V_{(s-3) \omega_{\beta}}$, the finite dimensional representation of $G$ with highest weight $(s-3) \omega_{\beta}$, where $\omega_{\beta} = 3 \alpha + 2 \beta$ is the fundamental weight associated to $\beta$. $I(\varepsilon,s)_2$ is isomorphic to the quaternionic discrete series representation $\pi_k^{\mathrm{GW}}$. This can be seen using e.g. the embedding of finite-dimensional representations into induced representations as in \cite[Lemma 8.5.7]{Wallach73}.
\end{remark}

\appendix

\section{$\fsu(2)$ tensor products}\label{appendix}

We recall the tensor product decomposition of representations of $\SU(2)$. For this, we realize the unique irreducible $(m+1)$-dimensional representation $\Gamma_m$ on the space of polynomials $f:\CC\to\CC$ of degree $\le m$ with the group action
$$ (g\cdot f)(z) = (-cz+d)^{m} f((az-b)(-cz+d)^{-1}), \qquad z\in\CC,g=\begin{pmatrix}a&b\\c&d\end{pmatrix}^{-1}\in\SU(2). $$
The invariant norm $\Vert\cdot\Vert_m$ on $\Gamma_m$ is given by
$$
\Vert f\Vert_m^2=C_m \int_{\mathbb C} |f(z)(1+|z|^2)^{-\frac{m}{2}}|^2 
(1+|z|^2)^{-2}dm(z),
$$
where $dm(z)$ denotes Lebesgue measure on $\CC\simeq\RR^2$ and the constant $C_m=\frac{m+1}{\pi}>0$ is chosen such that $\Vert 1\Vert_m=1$. Formulas for the $\SL(2,\mathbb{R})$-case appeared in \cite{Za94}. The norm of the monomial $z^j$ is given by
\begin{equation}\label{eq:NormMonomial}
    \Vert z^k\Vert_m^2 = {m\choose k}^{-1},
\end{equation} 
and hence the reproducing kernel of $\Gamma_m$ is the function
$$
(z,w)\mapsto\sum_{k=0}^m \binom{m}k (z\bar w)^k = (1+z\bar w)^m.
$$

We write the monomial basis $\{1, z, \ldots, z^m\}$ as $\{\xi_m^m,\xi_{m-2}^m,\ldots,\xi_{-m}^m\}$ with
\begin{equation}
  \label{xi_a}
\xi_a^m(z) = z^{\frac{m-a}2}, \quad -m \le a\le m,   a\equiv m 
\pmod 2.
\end{equation}
Then a short computation shows that
\begin{equation*}
    d\Gamma_m\begin{pmatrix}1&0\\0&-1\end{pmatrix}\xi_a^m = a\xi_a^m, \qquad d\Gamma_m\begin{pmatrix}0&1\\0&0\end{pmatrix}\xi_a^m = \frac{m-a}{2}\xi_{a+2}^m, \qquad d\Gamma_m\begin{pmatrix}0&0\\1&0\end{pmatrix}\xi_a^m = \frac{m+a}{2}\xi_{a-2}^m.
\end{equation*}
In particular, $\xi_m^m=1$ a highest weight vector with respect to the Borel subgroup of upper triangular matrices in $\SL(2,\CC)$.\\

The projection from the tensor product $\Gamma_m\otimes\Gamma_n$ onto one of the irreducible subrepresentations $\Gamma_{m+n-2k}$, $0\leq k\leq\min(m,n)$, is given by the Rankin--Cohen bracket
\begin{multline*}
    \RC_{m, n; k}: \Gamma_m\otimes \Gamma_n \to \Gamma_{m+n-2k},\\
    \RC_k: f(z)\otimes g(w)\mapsto\sum_{j=0}^k (-1)^j \binom{k}{j}\frac{1}{(-m)_j  (-n)_{k-j}}\partial_z^j f \partial_w^{k-j}|_{z=w}.
\end{multline*}
Its dual operator $\RC_k^\ast$ with respect to the invariant inner products mentioned above has the following expression:
\begin{multline*}
    \RC_k^\ast: \Gamma_{m+n-2k} \to \Gamma_m\otimes\Gamma_n, \quad \RC_k^\ast f(z, w) =C_{m +n-2k}(z-w)^k\int_{\mathbb C}f(x)  (1+z\bar x)^{m -k}(1+w\bar x)^{n -k}\\
    \times(1+|x|^2)^{-(m+n-2k)-2}dm(x).
\end{multline*}
In fact, it is easy to prove that the integral operator $\RC_k^\ast$ is intertwining, and then the formula for $\RC_k=(\RC_k^\ast)^\ast$ is obtained by differentiating the reproducing formula for $\Gamma_m\otimes\Gamma_n$ and evaluating at the diagonal (see e.g. \cite{Peng-Zhang-jfa} for the non-compact case and for general Hermitian symmetric spaces).

\begin{lemm+}
    For every $0\leq k\leq\min(m,n)$ we have
    $$ \RC_k^\ast 1 = (z-w)^k $$
    and the square norm of $(z-w)^k$
    is
    $$ \frac{k! (m+n-2k+2)^-_{k}}{(m)^-_k (n)^-_{k}}, $$
    where $(a)^-_k=a(a-1)\cdots (a-k+1)$ is the product of the $k$ descending factor starting with $a$. In particular $(\frac{(m)_k^-(n)_k^-}{k!(m+n-k+1)_k^-})^{\frac 12} \RC^\ast_k$ is an isometry and $(\frac{(m)_k^-(n)_k^-}{k!(m+n-k+1)_k^-})^{\frac 12} \RC_k$ is a partial isometry.
\end{lemm+}

\begin{proof}
    Since $\RC_k^\ast$  preserves the $U(1)$-weight, $\RC_k^\ast1$ is a sum
    $$ \RC_k^\ast1(z,w)=\sum_{j=0}^k c_j z^jw^{k-j}. $$
    We find $c_j$ by taking the inner product with $z^jw^{k-j}$,
    $$ c_j\Vert z^j\Vert_m^2\Vert w^{k-j}\Vert_n^2 = \langle\RC_k^*1,z^jw^{k-j}\rangle = \langle 1,\RC_k(z^jw^{k-j})\rangle = (-1)^j\frac{k!}{(-m)_j(-n)_{k-j}}, $$
    and using \eqref{eq:NormMonomial}, so constant $c_j =(-1)^{k-j}\binom kj$ and the claimed formula holds.
    Now we compute $\RC_k\RC_k^\ast 1$ using the definition of $\RC_k$:
    $$ \RC_k\RC_k^\ast 1 = \RC_k[(z-w)^k] = \sum_{j=0}^k(-1)^j{k\choose j}\frac{1}{(-m)_j(-n)_{k-j}}\cdot(-1)^{k-j}k!. $$
    The latter sum can be written as hypergeometric series
and can further be 
evaluated by Gauss' summation formula,
    $$ = {n\choose k}^{-1}\sum_{j=0}^k\frac{(-k)_j(n-k+1)_j}{(-m)_jj!} =\frac{k!(m+n-k+1)_k^-}{(m)_k^-(n)_k^-}. $$
    The rest follows since $\RC_k$ is intertwining.
\end{proof}  

\bibliographystyle{amsplain}
\bibliography{bibdb}

\end{document}